\documentclass[11pt,a4paper]{article}
\usepackage[english]{babel}
\usepackage{amsmath}
\usepackage{amssymb,amsbsy}
\usepackage{amsthm}
\usepackage[utf8]{inputenc}
\usepackage[T1]{fontenc}
\usepackage{enumerate}
\usepackage{mathrsfs}
\usepackage{bbm}
\usepackage{bm}

\usepackage{color}

\usepackage{hyperref}

\usepackage{nicefrac}

\newtheorem{tw}{Theorem}[section]
\newtheorem{lm}[tw]{Lemma}
\newtheorem{wn}[tw]{Corollary}
\newtheorem{pr}[tw]{Proposition}

\theoremstyle{definition}
\newtheorem{df}{Definition}[section]
\newtheorem{uw}[tw]{Remark}

\newcommand{\R}{\mathbb{R}}
\newcommand{\Z}{\mathbb{Z}}
\newcommand{\N}{\mathbb{N}}
\newcommand{\T}{\mathbb{T}}

\newcommand{\cQ}{\mathcal{Q}}

\newcommand{\Q}{\mathbb{Q}}

\newcommand{\cB}{\mathcal{B}}
\newcommand{\cC}{\mathcal{C}}
\newcommand{\cP}{\mathcal{P}}

\newcommand{\bea}{\begin{eqnarray}}
  \newcommand{\eea}{\end{eqnarray}}
  \newcommand{\beab}{\begin{eqnarray*}}
  \newcommand{\eeab}{\end{eqnarray*}}
\renewcommand{\a}{\alpha}
  \newcommand{\be}{\begin{equation}}
  \newcommand{\ee}{\end{equation}}
  
\providecommand{\noopsort}[1]{} 

\title{Slow entropy for some smooth flows on surfaces}
\author{Adam Kanigowski}
\begin{document}
\baselineskip=14pt \maketitle

\begin{abstract}We study slow entropy in some classes of smooth mixing flows on surfaces. The flows we study can be represented as special flows over irrational rotations and under roof functions which are $C^2$ everywhere except one point (singularity). If the singularity is logarithmic asymmetric (Arnol'd flows) we show that in the scale $a_n(t)=n(\log n)^t$ slow entropy equals $1$ (the speed of orbit growth is $n\log n$) for a.e. irrational $\a$. If the singularity is of power type ($x^{-\gamma}, \gamma\in (0,1)$) (Kochergin flows) we show that in the scale $a_n(t)=n^t$ slow entropy equals $1+\gamma$ for a.e. $\a$.\\
We show moreover that for local rank one flows, slow entropy equals $0$ in the $n(\log n)^t$ scale and is at most $1$ for scale $n^t$. As a consequence we get that a.e. Arnol'd and a.e Kochergin flow is never of local rank one. 
\end{abstract}
\tableofcontents

\section{Introduction}

\indent Smooth flows on surfaces\footnote{By a surface we always mean a compact connected orientable  $2$ dimensional manifold without boundary.} stand as one of the main class of study in dynamical systems. Dimension $2$ is the lowest in which we can observe some non-trivial ergodic and spectral properties, i.e. weak-mixing, mixing, decay of correlation. Indeed, in dimension $1$ every smooth flow is conjugated to a linear flow (which has discrete spectrum). Smooth surface flow has entropy $0$. This is a consequence Pesin formula \cite{pesin}. One of the central ergodic features describing chaoticity of the system (in the $0$ entropy case) is mixing. If a smooth surface flow has no fixed points, then by the Lefschetz formula the surface is a two-dimensional torus and the flow is a smooth time-change of a linear flow. In this case mixing never holds \cite{katokrig} (although weak-mixing holds for some smooth time changes with Liouvillean frequencies, \cite{sklover}). 

Therefore if one wants to obtain mixing examples in the class of smooth flows on surfaces one needs to consider flows with fixed points. Such examples where first shown to exist by Kochergin \cite{kocherginmixing} provided the existsence of a degenerate fixed point (Kochergin flows). If all fixed points are non-degenerate, and there are saddle connections, mixing was shown to hold in the ergodic component of the flow by Khanin and Sinai, \cite{sinaikhanin} if the transformation on the Poincare section is an irrational rotation (Arnol'd flows), for a full measure set of frequencies, and by Ulcigrai, \cite{ulcigraimixing}, if the transformation is an IET (for a full measure set of IET's). If all fixed points are non-degenerate and there are no saddle connections, then the flow is typically not mixing, \cite{lemanczykabsence,ulcigraiabsence}. However, Chaika and Wright \cite{ChW} showed the existence of a mixing flow with no saddle connections and non-degenerate fixed points on a genus $5$ surface. Recently, Fayad and the author \cite{FK} showed that in genus $1$ almost every Arnol'd flow and some Kochergin flows are mixing of all orders. This result was strenghtened in \cite{kku} to almost every Arnol'd flow in any genus.  
Moreover, Fayad, Forni and the author, \cite{FFK} ,showed that some Kochergin flows (with high order of degeneracy of the saddle) on the two torus have Lebesgue spectrum (which indicates stronger chaoticity of the system than mixing).

Smooth flows with singularities are believed to be systems of intermediate (polynomial growth). Indeed, on the one hand they have zero entropy, but on the other hand the presence of a fixed point produces stretching which results in fast divergence of nearby orbits when they get near the singularity. It is natural, that the orbit growth should depend on the order of degeneracy of the saddle. A very useful tool in making the notion of polynomial (or superlinear) growth precise is { \em slow entropy} introduced by Katok and Thouvenot in \cite{Kat-Tho}. It allows to make the (polynomial) orbit growth an isomorphism invariant (see Section \ref{sec.def}).
 For $0<\beta\leq 1$, we will also introduce the notion of $\beta$-slow entropy (see Section \ref{sec.def}) which measures the orbit growth on a portion of space of measure $=\beta$ (we denote the $\beta$-slow entropy by  $h_s^\beta$). The $\beta$ slow entropy is useful when dealing with systems of {\em local rank one}.
 To define slow entropy (or $\beta$-slow entropy) one needs a scale (which should be the expected orbit growth in the system).  The two scales that we use are $a_n(t)=n(\log n)^t$ and $a_n(t)=n^t$. As we will explain below, the first one works for Arnol'd flows and the second one for Kochergin flows. Before we state our main theorems, we need to specify the flows we will deal with.

Smooth flows that we will consider have representations as special flows over irrational rotation of the circle and under the roof function $f\in C^2(\T\setminus\{0\})$ which satisfies 
\begin{equation}\label{asu}\lim_{x\to 0^+}\frac{f(x)}{h(x)}=A_1\;\; \text{and}\;\; \lim_{x\to 0^-}\frac{f(x)}{h(1-x)}=B_1, \text{ where } A_1,B_1>0;
\end{equation}

 \begin{equation}\label{asu2}\lim_{x\to 0^+}\frac{f'(x)}{h'(x)}=-A_2\;\; \text{and}\;\; \lim_{x\to 0^-}\frac{f'(x)}{h'(1-x)}=B_2,\text{where} A_2,B_2>0;
\end{equation}

\begin{equation}\label{asu3}\lim_{x\to 0^+}\frac{f''(x)}{h''(x)}=A_3\;\; \text{and}\;\; \lim_{x\to 0^-}\frac{f''(x)}{h''(1-x)}=B_3,\text{where} A_3,B_3>0;
\end{equation}
where $h$ is
\begin{enumerate}
    \item  $-\log x$, then we assume additionally that $A_i+B_i\neq 0$, $i=1,2,3$ and the corresponding special flow $(T_t^f)$ is called an {\em Arnol'd flow}.
	 \item $x^{-\gamma}, 0<\gamma<1$, then the special flow $(T_t^f)=(T_t^{f,\gamma})$ is called a {\em Kochergin flow}. 
\end{enumerate}
We will now describe the full measure sets for which we can prove our theorems.
Let 
$$
\mathcal{D}:=\{\a\in\R\setminus\Q\;:\;q_{n+1}\leq C(\a)q_n\log q_n(\log n)^2\}. 
$$
It follows from Khinchin theorem, \cite{Khi}, that $\lambda(\mathcal{D})=1.$ 
For $\a\in \R\setminus\Q$, let $K_\a:=\{n\;:\;q_{n+1}\leq q_n\log^{7/8}q_n\}$ and define
$$
\mathcal{E}:=\{\a\in\R\setminus\Q\;:\;\sum_{i\notin K_\a} \log^{-7/8}q_i<+\infty \}.
$$
It is shown in \cite{FK} that $\lambda(\mathcal{E})=1$. 
With the above definitions our main theorems are the following (see Section \ref{sec.def} for the precise definition of $h_s^\beta$):
\begin{tw}\label{main} Let $\beta\in (0,1]$ and  $a_n(t)=n(\log n)^t$. Then for every $\alpha\in \mathcal{D}\cap \mathcal{E}$ and the corresponding Arnol'd flow $(T_t^f)$, we have $h_s^\beta(T_t^f)=1$.
\end{tw} 
Second main theorem deals with Kochergin flows.

\begin{tw}\label{main2} Let $\beta\in (0,1]$ and  $a_n(t)=n^t$. Then for every $\alpha\in \mathcal{D}$ and every $\gamma\in(0,1)$ for the corresponding Kochergin flow $(T_t^{f,\gamma})$, we have $h_s^\beta(T_t^{f,\gamma})=1+\gamma$.
\end{tw}

For Theorem \ref{main} the diophantine condition on $\alpha$ is crucial. In Remark \ref{rem.liou} we explain why this is not true for Liouvillean irrationals.


The following proposition gives an upper bound on orbit growth for local rank one flows (see Section \ref{def.rank} for the definition of local rank one flow).
\begin{pr}\label{ranon} Let $g(n)$ be any sequence of positive numbers such that\\
$\lim_{n\to+\infty} g(n)=+\infty$. Let $a_n(t)=n(g(n))^t$ and $(T_t)$ be a $\beta$- rank one flow for some $\beta\in (0,1]$. Then $h_s^\beta(T_t)=0$.
\end{pr}

Theorems \ref{main}, \ref{main2} and Proposition \ref{ranon} have the following consequence:
\begin{wn}\label{nro}
Every Arnol'd flow with the frequency in $\mathcal{D}\cap \mathcal{E}$ is not of local rank one. Every Kochergin flow with the frequency in $\mathcal{D}$ is not of local rank one. Moreover any two Kochergin flows with frequencies in $\mathcal{D}$ and different exponents ($\gamma\neq \gamma'$) are not isomorphic. 
\end{wn}

To determine the orbit growth in Theorems \ref{main} and \ref{main2} one needs to study the Birkhoff sums of the derivative of $f$. The growth of the Birkhoff sums of the derivative in case of asymmetric logarithmic singularities is of order $n\log n$, in case of power singularities the growth is of order $n^{1+\gamma}$. This justifies the choice of the scale in Theorem \ref{main} and \ref{main2} and also gives the intuition on the upper bound of $h_s^\beta$ (we bound the number of Hamming balls by the number of Bowen (topological) balls). It is the lower bound which is more difficult and crucial i.e. one needs to show that the statistical orbit growth (with Hamming metric) is equal to the topological (with Bowen metric). This can be considered as {\em variational principle} for slow entropy in the above classes (in general, variational principle for slow entropy does not hold \cite{Park}). 

The strategies of the proof of the lower bound are different in Theorem \ref{main} and Theorem \ref{main2}. In the first one we rely on slow, uniform divergence of nearby orbits-- Ratner's property (\cite{ratner2}). It is based on some ideas in \cite{FK}, where Ratner's property (in a weaker form) is shown to hold for a.e. Arnol'd flow. However, in the case of Theorem \ref{main2}, Ratner's property holds only for a measure $0$ set of irrational rotations. Therefore we have to apply another strategy. The technique is based on polynomial divergence of orbits (in the direction of the flow) and equidistribution properties of the base (Denjoy-Koksma type estimates). A similar technique has been used by Ratner in \cite{ratner} for proving that the square of the horocycle flow is not loosely-bernoulli and very recently in \cite{KRV} to give examples  on $\T^4$ of smooth K-automorphism which are not Bernoulli.
\paragraph{Plan of the paper.} In Section \ref{sec.def} we give the definition of slow entropy, local rank one and some other needed definitions. In Section \ref{sec.par} we introduce the definition of the PD-property\footnote{Acronym for parabolic divergence.}. This property, which is also of independent interest, is a crucial tool in the proof of Theorem \ref{main}.  In Section \ref{sran} we give the proof of Proposition \ref{ranon}. In Section \ref{Arn} we first show that if an Arnol'd flow satisfies the PD-property, then $h_s^\beta(T_t^f)=1$ (in the scale $a_n(t)=n(\log n)^t$). Finally we show that for a.e. $\alpha$ the corresponding Arnol'd flow satisfies the PD-property. This proofs Theorem \ref{main}. Then in Section \ref{Koc} we prove  
Theorem \ref{main2}.
\paragraph{Ackwnoledgments} The author would like to thank Anatole Katok for suggestion to study slow entropy for smooth surface flows. 
The author would also like to thank Mariusz Lema{\'n}czyk and Jean-Paul Thouvenot for several discussions on the subject. The author is grateful to Daren Wei for corrections on the first draft of the paper.

\section{Notation and basic definitions}\label{sec.def}
We will consider measure-preserving flows $(T_t)_{t\in\R}$ acting on probability Borel spaces $(X,\cB,\mu)$.
\subsection{Slow entropy}
Slow entropy was introduced in \cite{Kat-Tho} for actions of discrete amenable groups. Following the construction from \cite{Kat-Tho}, we give the definition of the $\beta$-slow entropy, $\beta\in(0,1]$, for flows below. 

For $r\in \R_+$ and $k\in \N$ we define the following symbolic space
\begin{equation}\label{symb0}
\Omega_{k,r}:=\{(w_t)_{t\in[0,r]}\;:\; w_t\in\{1,\ldots,k\}\}.
\end{equation}
Let $\lambda$ denote the Lebesgue measure on $\R$ (sometimes we will also write $\lambda$ for the Haar measure on the circle, it will be clear from the context which measure are we dealing with). For $r\in \R_+$ and $x,y\in \Omega_{k,r}$, the {\em Hamming distance} ({\em Hamming metric}) of $x,y$ is defined in the following way :
$$
d_r(x,y):=\frac{r-\lambda(\{s\in[0,r]\;:\; x_s=y_s)}{r}.
$$ 

Let now $(T_t)_{t\in\R}$ be a flow acting on $(X,\cB,\mu)$ and let 
$\cP=\{P_1,...,P_k\}$ be a finite measurable partition of $X$. For $r\in \R_+$ we define the {\em coding map} $\phi_{\cP,r}:X\to \Omega_{k,r}$,
\begin{equation}\label{par0}
 \phi_{\cP,r}(x)=(N_t(x))_{t\in[0,r]},\text{ where }   N_{t}(x)=i\;\;\; \text{   iff   }\;\;\; T_{t}x\in P_i.
\end{equation}
$\phi_{\cP,r}(x)$ is also called the {\em $\cP,r$-name} of $x\in X$. The space $\Omega_{k,r}$ is a metric space (equipped with the Hamming metric $d_r$) and there is also a natural probability measure on $\Omega_{k,r}$, $\nu:=(\phi_{\cP,r})_\ast\mu$, associated with the dynamics. Let us fix $\epsilon>0$ and define the following quantity:
\begin{multline}\label{bal1}
S^r_\cP(\epsilon,\beta):= \text{ the minimal number of } d_r-\text{ balls of radius }\epsilon \text{ whose union has }\\
\nu-\text{measure greater than }\beta-\epsilon.
\end{multline}
We denote $S^r_\cP(\epsilon,1)$ by $S^r_\cP(\epsilon)$. For $x\in X$ we denote $B^{r}_{\cP}(x,\epsilon)$ the $\epsilon$ ball around $x$.
Note that by the definition of the measure $\nu$ the above quantity depends strongly on the dynamics of the flow $(T_t)_{t\in\R}$.

The next step is to define a family of sequences $a(n,t)_{n\in \N,t\in \R_+}$ such that for fixed $t_0\in R_+$, we have $\lim_{n\to+\infty}a(n,t_0)=+\infty$.
The "scale" $a(n,t)$ will measure the asymptotics growth of orbits, it should be chosen in connection to the dynamics. The two scales which will be useful for us are $n^t$ and $n(logn)^t$. For $\beta\in(0,1]$ we define
\begin{equation}\label{epdel}
A(\cP,\epsilon,\beta):=\sup\{t>0\;:\;\liminf_{r\to+\infty}
\frac{S^r_\cP(\epsilon,\beta)}{a([r],t)}>0\}.
\end{equation}
One can also define $A(\cP,\epsilon,\beta)$ with $\limsup$ instead of $\liminf$, but for our purposes it is more convenient to work with $\liminf$. The function $S^r_\cP(\epsilon,\beta)$ is non-increasing in $\epsilon$, therefore we can define 
\begin{equation}\label{alm}
A^\beta(\cP):=\lim_{\epsilon\to 0}A(\cP,\epsilon,\beta).
\end{equation}
Then the $\beta$-slow entropy of $(T_t)_{t\in \R}$ is given by 
\begin{equation}\label{hsb}
h_s^\beta(T_t)=\sup_{\cP}A^\beta(\cP).
\end{equation}
We denote $h_s^1(T_t)$ by $h_s(T_t)$ and call simply the slow entropy of $(T_t)_{t\in\R}$.

Note that the slow entropy ($\beta$-slow entropy) of a system depends on the "scale" $a_n(t)$, but since we will always fix a scale on the beginning we omit it in the notation of the slow entropy ($\beta$-slow entropy).
 Note that 
$$0\leq, h^\beta_s(T_t)\leq +\infty\text{ for }0<\beta\leq 1.$$ If $(T_t)_{t\in\R}$ is ergodic
and $a(n,t)=e^{nt}$, then $h_s(T_t)$ is just the entropy of the flow, \cite{Kat-Tho}.\\

Recall that a partition $\cP$ is called a {\em generator} if the minimal $(T_t)_{t\in\R}$ invariant $\sigma$-algebra containing $\cP$ is the whole $\sigma$-algebra $\cB$. A sequence of partitions $(\cP_n)_{n\in\N}$ is called {\em generating} if it converges to a partition into points. The following proposition was shown to hold in \cite{Kat-Tho} for discrete groups and for slow entropy, however the proof in the case of flows and $\beta$-slow entropy is completely analogous.
\begin{pr}(\cite{Kat-Tho}, Proposition 1.)\label{calc} If $\cP_m$ is a generating sequence of partitions, then for every $\beta\in(0,1]$
$$ h^\beta_s(T_t)=\lim_{m\to+\infty}A^\beta(\cP_m).
$$
\end{pr}
It follows by the above proposition, that if $\cP$ is a generator, then
$h^\beta_s(T_t)=A^\beta(\cP)$.

\subsection{Special flows}
Denote by $\lambda$ the Lebesgue measure on $\R$.
 Let $T:(X,\cB,\mu)\to (X,\cB,\mu)$ be an automorphism and $f\in L^1(X,\cB,\mu)$, $f>0$. Then the special flow $\mathcal{T}_f:=(T_{t}^f)_{t\in\R}$
 acts on the space $(X^f,\cB^f,\mu^f)$, where
$X^f:=\{(x,s)\;:\;x\in X, 0\leq s<f(x)\}$ and $\cB^f:=\cB\otimes \cB(\R)$ and $\mu^f:=\mu\times \lambda$. Under the action of the flow $\mathcal{T}_f$ each point in $X^f$ moves vertically with unit speed and we identify the point $(x,f(x))$ with $(Tx,0)$. More precisely, if $(x,s)\in X^f$ then
$$T_t^f(x,s)=(T^{N(x,t)}x,s+t-f^{(N(x,t))}(x)),$$
where $N(x,t)\in\Z$ is unique such that
$$f^{(N(x,t))}(x)\leq s+t<f^{(N(x,t)+1)}(x).$$ 
For flows we consider, $T:(\T,\cB,\lambda)\to(\T,\cB,\lambda)$, $Tx=x+\alpha\;\; mod \;1$ and $f$ is given by \eqref{asu},\eqref{asu2} and \eqref{asu3}.
There is a natural prodcut metric $\rho^f$ on $\T^f$ given $by \rho^f((x,s),(y,s'))=
\|x-y\|+|s-s'|$. The following pseudo-metric will be crucial for estimating the number of Hamming balls: 
\begin{multline*}
d^f((x,s),(y,s'))=
\min(\rho^f((x,s),(y,s')),
\|x+\alpha-y\|+|f(x)-s+s'|,\\
\|x-(y+\alpha)\|+|f(y)-s'+s|).
\end{multline*}
We have the obvious inequality $d^f\leq \rho^f$. Using the pseudo-metric $d^f$ will make some computations later easier, i.e. $d^f(T_t(x,s),T_t(y,s'))<\epsilon$ for $t\in[A,B]$ means that the difference of Birkhoff sums for $x$ and $y$ is small and we avoid the problem of $T_tx$, $T_ty$ being on "different sides" of the graph of $f$ (which means that they are not close in $\rho^f$). For the rest of the paper we will work with the pseudo-metric $d^f$ (we will not use the triangle inequality).
\subsubsection{Denjoy-Koksma estimates}
The following lemma is a simple consequence of the Denjoy-Koksma inequality We make a standing assumption that $\int_\T fd\lambda=1$.
For $x\in \T$ denote
$$x^M_{min}=\min_{j\in [0,M)}d(x+j\a,0).$$
Let in \eqref{asu}, \eqref{asu2},\eqref{asu3}, $h(x)=x^{-\gamma}$ and $A_i=B_i=1$, $i=1,2,3$. The following lemma is a straightforward consequence of the Denjoy-Koksma inequality and Ostrovski expansion along the sequence of denominators (see e.g. \cite{FFK}, Lemma 3.1.)
\begin{lm}\label{koksi2} Let $\alpha\in \mathcal{D}$.  For every $x\in \T$ and every $M\in \Z$ $|M|\in [q_s,q_{s+1}]$, we have 
\begin{equation}\label{koks3}
q_s-4q_s^{1-\gamma}\leq f^{(q_s)}(x) \text{ and } |f^{(M)}(x)-M|\leq M^{1-\gamma}\log^4M+\log^3Mf(x^M_{min}) 
\end{equation} 
\begin{equation}\label{koks4}
f'(x^M_{min})-8q_{s+1}^{1+|\gamma|}\leq
|f'^{(M)}(x)|\leq f'(x^M_{min})+8q_{s+1}^{1+|\gamma|}
\end{equation}
and
\begin{equation}\label{koks5}
f''(x^M_{min})\leq f''^{(M)}(x)\leq
f''(x^M_{min})+8q_{s+1}^{2+|\gamma|}.
\end{equation}
\end{lm}

The second lemma deals with asymmetric logarithmic singularities.

Let in \eqref{asu}, \eqref{asu2}, \eqref{asu3}, $h(x)=-\log x$,  $A_1=A_2=A_3=1$ and $B_1=B_2=B_3=2$.

\begin{lm}\label{koksi}   For every $x\in \T$ and every  $n\geq\ n_0$
$$
(1-10^{-3})q_n\log q_n - 4|f'(x^{q_n}_{min})|\geq f'^{(q_n)}(x)\leq 
(1+10^{-3})q_n\log q_n + 4|f'(x^{q_n}_{min})|
$$
\end{lm}
The proof of Lemma \ref{koksi} is analogous to the proof of Lemma \ref{koksi2}.

Let $(\kappa_n)$ be an increasing sequence growing slowly to $+\infty$ ($\kappa_n=\log n$ for example).

\begin{lm}\label{goodcon}Let $\alpha \in \mathcal{D}$. There exists $R_0\in \N$ such that for every $n\geq n_0$, $R\in \Z$, $|R|\geq R_0$ and every $x,y\in \T$ satisfying 
\begin{equation}\label{empt2}
\left(\bigcup_{i=0}^{R}T^{i}[x,y]\right)\cap\left[-\frac{\kappa_n}{q_n\log q_n},\frac{\kappa_n}{q_n\log q_n}\right]=\emptyset,
\end{equation}
for every $r\in \Z$,  $rR>0$, $|r|\in [0,\frac{q_n}{1000}]$, we have
\begin{equation}\label{birk.cont}
|f^{(r)}(x)-f^{(r)}(y)|\leq \frac{1}{500}q_n\log q_n\|x-y\|,
\end{equation}
and for every $r\in \Z$, $rR>0$, $|r|\in[\frac{q_n}{1000},R]$, we have
\begin{equation}\label{birk.cont2}
\frac{11}{10}|r|\log |r| \|x-y\|>|f^{(r)}(x)-f^{(r)}(y)|\geq \frac{9}{10}|r|\log |r| \|x-y\|.
\end{equation}
\end{lm}
\begin{proof}  By \eqref{empt2}, for every $r\in [0,R]$ there exists $\theta_r\in [x,y]$ such that 
$$
|f^{(r)}(x)-f^{(r)}(y)|=|f'^{(r)}(\theta_r)|\|x-y\|.
$$
We will conduct the proof for $r>0$, the proof for $r<0$ is analogous. Fix $r\in [0,R]$ and let $r=\sum_{i=0}^jb_iq_i$, where $b_i\leq\frac{q_{i+1}}{q_i}\leq \log^2q_i$. We can WLOG assume that 
\begin{equation}\label{mint}
\min_{0\leq w< r}d(\theta_r+w\alpha,0)=\theta_r.
\end{equation}
Indeed, we have (for $w$ which gives the minium in \eqref{mint}) 
$$
f'^{(r)}(\theta_r)=f'^{(-w)}(\theta_r+w\alpha)+f'^{(r-w)}(\theta_r+w\alpha),
$$
and we do the estimates below separately for $f'^{(-w)}(\theta_r+w\alpha)$ and $f'^{(r-w)}(\theta_r+w\alpha)$. 
We have
$$
f'^{(r)}(\theta_r)=\sum_{k=0}^{j} f'^{(b_{j-k}q_{j-k})}(\theta_r+\sum_{s=j-k+1}^{j}b_sq_s\alpha),
$$
(with $b_{j+1}q_{j+1}=0$). Denote $S_{j,k}= \sum_{s=j-k+1}^{j}b_sq_s\alpha$.
 Fix $k\in[0,j]$ . We have
$$
f'^{(b_{j-k}q_{j-k})}(\theta_r+S_{j,k})=
\sum_{s=0}^{b_{j-k}-1} f'^{(q_{j-k})}(\theta_r+S_{j,k}+sq_{j-k}\alpha).
$$
Denote $\Theta_{j,k}(s)=\theta_r+S_{j,k}+sq_{j-k}\alpha$ and consider  $(\Theta_{j,k}(s))^{q_{j-k}}_{min}$, $s=0,...,b_{j-k}-1$. They are contained in the interval $[-\frac{1}{2q_{j-k}},\frac{1}{2q_{j-k}}]$ and form a progression with spacing $\|q_{j-k+1}\alpha\|$. Moreover by \eqref{mint}, we have
$$
\min_{s=0,...,b_{j-k}-1}d\left((\Theta_{j,k}(s))^{q_{j-k}}_{min}, 0\right)\geq \frac{1}{2q_{j-k+1}}
$$
and by \eqref{empt2} (if $j=n$)
\begin{equation}\label{fhg}
\min_{s=0,...,b_{n}-1}d\left((\Theta_{n,0}(s))^{q_{n}}_{min}, 0\right)\geq \frac{\kappa_n}{q_n\log q_n}
\end{equation}

Therefore and by Lemma \ref{koksi}, we get 
\begin{multline*}
b_{j-k}q_{j-k}\log b_{j-k}q_{j-k} +4(2q_{j-k+1}+q_{j-k+1}\log b_{j-k}) \geq\\	f'^{(b_{j-k}q_{j-k})}(\theta_r+S_{j,k})\geq\\ b_{j-k}q_{j-k}\log b_{j-k}q_{j-k}- 4(2q_{j-k+1}+q_{j-k+1}\log b_{j-k}).
\end{multline*}
Moreover if $j=n$, by \eqref{empt2}, we get
$$
\frac{101}{100}b_nq_n\log(b_nq_n)\geq f'^{(b_{n}q_{n})}(\theta_r)\geq \frac{99}{100}b_nq_n\log(b_nq_n). 
$$
Then \eqref{birk.cont} and \eqref{birk.cont2} follow by summing up over $k=0,...,j$.
\end{proof}

Let $\omega_n=\log(\log n)$.

\begin{wn}\label{wniosek2} Assume $\alpha\in \mathcal{D}$.  For every $R\in \Z$ sufficiently large and for every $x,y\in \T$ such that 
$$
\left(\bigcup_{i=0}^{R}T^{i}[x,y]\right)
\cap\left[-\frac{1}{2|R|\omega_{|R|}},
\frac{1}{2|R|\omega_{|R|}}\right]=\emptyset,
$$
we have
$$
|f^{(R)}(x)-f^{(R)}(y)|< |R|\log |R|\omega_{|R|}^4\|x-y\|.
$$
\end{wn}
\begin{proof} We will conduct the proof for $M>0$ (the proof in case $M<0$ is analogous). Let $k\in \N$ be unique such that $(k-1)q_n\leq M<kq_n$. By assumptions on $\alpha$ it follows that $k\leq \log M\omega_{M}^2$. 
By assumptions, for some $\theta\in[x,y]$, we have 
$$
|f^{(M)}(x)-f^{(M)}(y)|=|f'^{(M)}(\theta)|\|x-y\|.
$$
By \eqref{asu2} it follows that for some $C>0$
$
|f'^{(M)}(\theta)|\leq C|f'^{(kq_n)}(\theta)|$.  Moreover, by assumptions on $x,y$, by \eqref{koksi} for $\theta$ and cocycle identity, we have
$|f'^{(kq_n)}(\theta)|\leq kq_n\log q_n +kf'(\theta^R_{min})$. But $\theta^R_{min}\geq 2M\omega_M$. Putting all together we get 
$$
|f'^{(M)}(\theta)|\leq C M\log M+4M\log M\omega_M^3.
$$
This finishes the proof.
\end{proof}

\subsection{Rank one systems and systems of local rank one}\label{def.rank}
In this section we will introduce the notion of rank one and local rank one. There are several equivalent ways to define a rank one system (see \cite{Fen}). We will define rank properties in the language of special flows, \cite{fay}.
Let $(T_t)_{t\in\R}$ be an ergodic flow on $(X,\cB,\mu)$. For every $H\in \R_+$ and $1> \eta>0$  We can represent $(T_t)_{t\in\R}$ as a special flow over ergodic $S:(Y,\cC,\nu)$ and the a roof function $\phi$ satisfying:

\begin{enumerate}
    \item $\phi(y)<H$ for every $y\in Y$;	
	 \item  there exists a set $B=B(\eta,H)\in \cC$ with $\nu(B)>1-\eta$ and such that $\phi(y)=H$ for every $y\in B$.
\end{enumerate} 
We will denote the special representation of $(T_t)$ by $T_t^\phi$ which acts on\\ $(Y^f,C^f,\nu^f)$.
Fix a finite partition $\cP$ of $X$. The set $\bigcup_{s=0}^HT_s(B)$ is called a {\em tower for $(T_t)_{t\in\R}$}. For $0\leq s<H$ the set $T_s(B)$ is called a {\em level} of the tower. For $\epsilon>0$ we say that a level $T_s(B)$ is {\em $\epsilon$- monochromatic (for $\cP$)} if its $1-\epsilon$ proportion with respect to measure $(T_s)_\ast\nu$ is included in one atom of the partition $\cP$. A tower for $(T_t)_{t\in\R}$ is called {\em $\epsilon$-monochromatic (for $\cP$)} if $1-\epsilon$ proportion of levels (with respect to the Lebesgue measure $\lambda$ on $[0,H]$) is $\epsilon$-monochromatic.

\begin{df}\label{rankone} Fix $\beta \in(0,1]$. An ergodic flow $(T_t)_{t\in\R}$ acting on $(X,\cB,\mu)$ is an {\em $\beta$-rank one flow} if for every $\epsilon>0$ and every finite partition $\cP$ of $X$ there exists a tower for $(T_t)_{t\in\R}$ of measure greater than $\beta-\epsilon$ which is $\epsilon$-monochromatic for $\cP$.
\end{df}
An $1$-rank one flow is called a {\em rank one flow}.

\section{Parabolic divergence property}\label{sec.par}
In this section we will introduce a property which is characteristic for parabolic dynamics and is a useful tool for computing slow entropy of such systems. We will assume that $(T_t)_{t\in\R}$ is an ergodic flow on a metric space $(X,\rho)$ with Borel $\sigma$-algebra and Borel probability measure $\mu$. Let $d$ be a pseudo-metric on $X$. For 
$x,y\in X$ such that $d(x,y)<10^{-2}$  denote
\begin{multline}\label{ixy}I_{x,y}-\text{ the maximal time interval containing $0$ such that for every }t\in I_{x,y}, \\
d(T_tx,T_ty)<10^{-2}.
\end{multline}
Notice that for every $t\in I_{x,y}$ we have 
\begin{equation}\label{neweq}
I_{T_tx,T_ty}=I_{x,y}.
\end{equation}
\begin{df}\label{pd} $(T_t)_{t\in\R}$ is said to have {\em PD-property}\footnote{An acronym for parabolic divergence.} if there exist $c_0,c_1\in(0,1)$ such that for every $\epsilon>0$ there exists $Z=Z(\epsilon)$, $\mu(Z)>1-\epsilon$ and $\delta=\delta(\epsilon)$ such that for every $x,y\in Z$ $d(x,y)<\delta$, we have 
\begin{equation}\label{op1}
\left|\{t\in I_{x,y}\;:\; d(T_tx,T_ty)>c_0\}\right|>c_1|I_{x,y}|.
\end{equation}
\end{df}
Since $(X,d)$ is a polish space, there exists a sequence of compact sets $(K_n)_{n\in \N}$ such that $\mu(K_n)\to 1$ and $K_i\subset K_{i+1}$ for every $i\in \N$. For fixed $n\in \N$ let $\mathcal{q}_n$ be a partition of $K_n$ into a finite number of disjoint sets of diameter$\in (\frac{1}{n},\frac{2}{n})$. Then $\cQ_n=\{\mathcal{q}_n,K_n^c\}$ is a partition of $X$.\\
The following proposition shows that in the class of flows satisfying the PD-property, if two points are Hamming close , they have to be Bowen close on a substantial proportion of their orbits. This is the crucial tool for Theorem \ref{main}.

\begin{pr}\label{db} Let $(T_t)$ have PD-property (with constants $c_0,c_1$). Then for every (sufficiently small) $\eta>0$  there exist $n_0=n_0(\eta),R_0=R_0(\eta)\in \N$ and a set $V=V(\eta)\subset X$, $\mu(V)>1-\eta$ such that for every $n>n_0$, $R>R_0$ and every $x,y\in V$ if $d^{\cQ_n}_R(x,y)<\min(c_0^2,c_1^2)$ then there exists an interval $[A,B]\subset [0,R]$ such that $\lambda(B-A)>\frac{c_1R}{10}$ and $d(T_tx,T_ty)<10^{-2}$ for every $t\in[A,B]$.
\end{pr}  
\begin{proof} Fix $\eta\in(0,10^{-2}\min(c_0^2,c_1^2))$. Let $Z=Z(\eta^4)$, $\mu(Z)\geq 1-\eta^4$ and $\delta=\delta(\eta^4)$ come from
PD-property. Let $n_0=n_0(\eta,c_0,c_1)\in \N$ be such that $\mu(K_{n_0})>1-\eta^4$ and $n_0\geq \min(c_0^2,c_1^2,\delta)^{-10}$. Since $(T_t)_{t\in\R}$ is ergodic, there exists a set $V=V(\eta)\subset X$, $\mu(V)>1-\eta$ and a number $R_0=R_0(\eta)$ such that for $R>R_0$ and $x\in V$
\begin{equation}\label{birk4}
\frac{1}{R}\int_0^R\chi_{K_{n_0}}(T_tx)dt>1-\eta^2
\end{equation}
and
\begin{equation}\label{birk3}
\frac{1}{R}\int_0^R\chi_Z(T_tx)dt>1-\eta^2.
\end{equation}
 Let  $R>R_0$, $n\geq n_0$ and take any $x,y\in V$ with $d^{\cQ_n}_R(x,y)<\min(c_0^2,c_1^2)$. Define

$$U:=\{t\in[0,R]\;:\; T_tx,T_ty\in Z\text{ and } d(T_tx,T_ty)<\delta\}.$$
Notice that by the definition of $\cQ_n$, if for some $s\in [0,R]$, $T_sx,T_sy\in K_n$ and $T_sx,T_sy$ are in one atom of $\cQ_n$, then $d(T_sx,T_sy)<\frac{2}{n}\leq \frac{2}{n_0}\leq \delta$. Therefore, by \eqref{birk4} and \eqref{birk3} for $x,y\in V$ and the fact that $K_{n_0}\subset K_n$ it follows that
\begin{equation}\label{lt}
|U|\geq (1-\min(c_0^2,c_1^2)-2\eta)R.
\end{equation}
Consider the set 
$$C_R:=\{t\in[0,R]: d(T_tx,T_ty)<10^{-2}\}.$$
Then $U\subset C_R$ and $C_R$ is a union of  (disjoint) intervals $\in[0,R]$. 
By the definition of $I_{x,y}$ (see \eqref{ixy}) it follows that there exists $l\in \N$ and an increasing sequence $(t_i)_{i=1}^l\subset [0,R]^l$ such that for $i\neq j$, $I_{T_{t_i}x,T_{t_i}y}\cap I_{T_{t_j}x,T_{t_j}y}=\emptyset
$ and
\begin{equation}\label{neweq2}
U\subset C_R\subset \bigcup_{i=1}^l I_{T_{t_i}x,T_{t_i}y}
\;\text{  and  }\;
\bigcup_{i=2}^{l-1}I_{T_{t_i}x,T_{t_i}y}\subset C_R.
\end{equation}
By \eqref{neweq2},
$I_{T_{t_i}x,T_{t_i}y}\subset [0,R]$ for $i=2,...,l-1$. Notice that by the PD-property, for every $s\in U$ 
$$
|I_{T_{s}x,T_{s}y}\cap U|\leq (1-c_1)\left|I_{T_{s}x,T_{s}y}\right|.
$$
Moreover for every $i\in\{2,...,l-1\}$ if $s\in U\cap I_{T_{t_i}x,T_{t_i}y}$, then by \eqref{neweq} we have 
$$
|I_{T_{t_i}x,T_{t_i}y}\cap U|\leq (1-c_1) \left|I_{T_{t_i}x,T_{t_i}y}\right|.
$$

Hence, by \eqref{lt}, \eqref{neweq2} and \eqref{birk3}, we have 
\begin{multline*}
(1-\min(c_0^2,c_1^2)-2\eta)R\leq |U|\leq \sum_{i=1}^l|U\cap I_{T_{t_i}x,T_{t_i}y}|\leq
\left|[0,R]\cap I_{T_{t_1}x,T_{t_1}y}\right|+\\
\left|[0,R]\cap I_{T_{t_l}x,T_{t_l}y}\right|+ (1-c_1)
\sum_{i=2}^{l-1}|I_{T_{t_i}x,T_{t_i}y}|\leq\\
\left|[0,R]\cap I_{T_{t_1}x,T_{t_1}y}\right|+
\left|[0,R]\cap I_{T_{t_l}x,T_{t_l}y}\right|+ (1-c_1)R.
\end{multline*}
Therefore, $\left|[0,R]\cap I_{T_{t_1}x,T_{t_1}y}\right|>\frac{c_1R}{4}$ or  $\left|[0,R]\cap I_{T_{t_l}x,T_{t_l}y}\right|>\frac{c_1R}{4}$. We conclude by setting $[A,B]$ to be the longer of  $[0,R]\cap I_{T_{t_1}x,T_{t_1}y},[0,R]\cap I_{T_{t_l}x,T_{t_l}y}$. 
\end{proof}

\subsection{PD-property for special flows}
In this section we will state a condition which implies PD-property for special flows. Let $T:(X,d,\mu)\to (X,d,\mu)$ be an ergodic isometry and let $f\in L^1(X)$ be strictly positive. Assume for simplicity that $\int_X f(x)d\mu=1$. For $x,y\in X$ let 
\begin{multline}\label{jxy}
J_{x,y}:=\text{the maximal interval }\subset \Z \text{ such that } 0\in J_{x,y} \text{ and for every }\\
 n\in J_{x,y},\;\;|f^{(n)}(x)-f^{(n)}(y)|<1/50.
\end{multline}

\begin{pr}\label{cocy} If for every $\epsilon>0$ there exists $\delta''=\delta''(\epsilon)$ and a set $Z_1=Z_1(\epsilon)\subset X$, $\mu(Z_1)>1-\epsilon$, such that  
 for every $x,y\in Z_1$, $d(x,y)<\delta''$ there exists an interval $K=K_{x,y}\subset J_{x,y}$, $|K|\geq \frac{|J_{x,y}|}{300}$, such that for every $n\in K$
$$
|f^{(n)}(x)-f^{(n)}(y)|>\frac{1}{400},
$$
then $(T_t^f)_{t\in\R}$ has the 
PD-property (with the pseudo-metric $d^f$).
\end{pr}
\begin{proof} Let $c_0=10^{-4} \min(\inf_\T f, \alpha)$ and $c_1=10^{-4}$. We will show that $(T_t^f)$ has the PD-property with $c_0$ and $c_1$. Fix $\epsilon>0$. Let $N_0=N_0(\epsilon)$ and $A=A(\epsilon)$, $\mu(A)>1-\epsilon^3$ be such that for every $n\in \Z$, $|n|\geq N_0$ and every $x\in A$ 
\begin{equation}\label{ergth}
(1-10^{-4})n<f^{(n)}(x)<(1+10^{-4})n.
\end{equation}
Let $\delta'=\delta'(\epsilon)$ and $B=B(\epsilon)$, $\mu(B)>1-\epsilon^3$ be such that for every $x,y\in B$, $d(x,y)<\delta'$ and $|n|\leq N_0$
\begin{equation}\label{yeg1} 
|f^{(n)}(x)-f^{(n)}(y)|<10^{-4}.
\end{equation}
The existence of such $\delta'$ and $B$ follows from Egorov's theorem: since $f$ is measureable, all $f^{(i)}$, $|i|\leq N_0$, are uniformly continuous on a set of arbitrary large measure. We will now define $\delta$ and $Z$ for PD-property. Let $\delta''=\delta''(\epsilon^3)$ and $Z_1=Z_1(\epsilon^3)$ come from assumptions. 

Define $\delta:=\min(\delta',\delta'')$ and $Z:=\{(x,s)\in X^f\;:\; x\in A\cap B\cap Z_1\}$. We will show that PD-property holds for all $(x,s),(y,r)\in Z$, $d^f((x,s),(y,r))<\delta$. \\
Fix $(x,s),(y,r)\in Z$, $d^f((x,s),(y,r))<\delta$. 
Let $J_{x,y}=[a,b]\cap \Z$ and $K_{x,y}=[c,d]\cap \Z$. It follows by the definition of $J_{x,y}$ that $a\leq 0\leq b$, $[c,d]\subset [a,b]$ and $d-c\geq \frac{b-a}{300}$. Therefore 
$$
d-c>\frac{b-a}{300}\geq \frac{b}{300}\geq \frac{d+c}{600}.
$$
Moreover since $x,y\in B$, $d^f((x,s),(y,r))<\delta$ and by \eqref{yeg1} it follows that 
\begin{equation}\label{larg}\min(|a|,|b|,|c|,|d|)>N_0.
\end{equation}

 Notice that by the definition of $I_{(x,s),(y,r)}$ and $J_{x,y}$ we get
$$
I_{(x,s),(y,r)}\subset [f^{(a)}(x)-s,f^{(b)}(x)-s].
$$
Moreover by the definition of $K_{x,y}\subset J_{x,y}$, the fact that $d^f((x,s),(y,r))<\delta$ and the definition of $c_0$ it follows that  
$$
[f^{(c)}(x),f^{(d)}(x)]\subset\{t\in I_{(x,s),(y,r)}\;:\; d^f(T^f_t(x,s),T^f_t(y,r))>c_0\}.
$$
Therefore, since $x,y\in B$, by \eqref{larg} and \eqref{ergth}, we get
\begin{multline*}
|\{t\in I_{(x,s),(y,r)}\;:\; d^f(T^f_t(x,s),T^f_t(y,r))>c_0\}|\geq f^{(d)}(x)-f^{(c)}(x)\geq\\
 (d-c)- \frac{d+c}{1200}\geq \frac{d-c}{2}\geq \frac{b-a}{600}\geq \frac{\max(|b|,|a|)}{600}\geq \\
\frac{\max(|f^{(b)}(x)|,|f^{(a)}(x)|)}{700}\geq \frac{f^{(b)}(x)-f^{(a)}(x)}{1400}\geq \frac{|I_{(x,s),(y,r)}|}{1400}\geq c_1|I_{(x,s),(y,r)}|.
\end{multline*}
This finishes the proof.

\end{proof}

\section{Slow entropy of rank one systems, proof of Proposition \ref{ranon}}\label{sran}
Fix a measurable partition $\cP=\{P_1,...,P_k\}$ of $X$. Let $\epsilon_n\to 0$ and let $(T_n)_{n\in\N}$ be a sequence of towers $\epsilon_n$-monochromatic for $\cP$ (see Definition \ref{rankone}). Denote $(B_n)_{n\in\N}$ and $(H_n)_{n\in\N}$ the sequences of bases and heights for $(T_n)_{n\in\N}$. For simplicity we will denote the measure on $B_n$ by $\nu$ (although the measure depends on $n$). The following lemma implies Proposition \ref{ranon}(see \eqref{bal1}):

\begin{lm}\label{nub} Fix $\epsilon>0$. For $n$ sufficiently large we have 
$$
S_\cP^{\epsilon^2 H_n}(\epsilon,\beta)<\frac{H_n}{t_0(\epsilon)}
$$
for some $t_0(\epsilon)$.
\end{lm}

Before we give the proof of Lemma \ref{nub} let us show how it implies Proposition \ref{ranon}.

\begin{proof}[Proof Proposition \ref{ranon}:] Let us fix a generator $\cP$. By Proposition \ref{calc} it is enough to show that $A(\cP,\epsilon,\beta)=0$ (see \eqref{epdel}) for sufficiently small $\epsilon>0$.
To show this, it is enough to show that there exists a sequence $(r_n)_{n\in \N}$, $\lim r_n=+\infty$, such that for every $t>0$ we have 
$$
\lim_{n\to +\infty}\frac{S^{r_n}_\cP(\epsilon,\beta)}{a([r_n],t)}=0.
$$
It is enoguh to take $r_n:=\epsilon^2 H_n$, use Lemma \ref{nub} and $g(n)\to +\infty$.
\end{proof}

So it remains to proof Lemma \ref{nub}, which will follow by the two following lemmas:
\begin{lm}\label{mesflow} For every $\epsilon>0$ there exist $\R_+\ni t_0(\epsilon)=t_0$ such that for every $n\geq 1$ there exists a set $W_{n,t_0}\subset B_n\times [0,t_0]$, $\nu\times \lambda(W_{n,t_0})>(1-\epsilon)t_0\nu(B_n)$, such that for every $(y,t)\in W_{n,t_0}$ we have 
$$
d_{H_n}(y,T_ty)<\epsilon.
$$
\end{lm}
The proof of  Lemma \ref{mesflow} uses only the existence of {\em Rokhlin towers} for $(T_t)$ which is true for every measurable flow. We have also the following lemma, which uses the rank one structure and can be found in \cite{KS}, Lemma 2.1.
\begin{lm}\label{eva} For every $\epsilon>0$ there exists $n_0$ such that for all $n\geq n_0$ there exists $D_n\subset B_n\subset$, $\nu(D_n)>(1-\epsilon)\nu(B_n)$ and such that for all $x,y\in D_n$, 
$$d_{H_n}(x,y)<\epsilon.$$
\end{lm}

Before we prove Lemma \ref{mesflow}, let us show how the two above lemmas imply Lemma \ref{nub}.

\begin{proof}[Proof of Lemma \ref{nub}] Fix $\epsilon>0$. We use Lemmas \ref{mesflow} and \ref{eva} for $\epsilon^4$. Define $V_n:=W_{n,t_0}\cap (D_n\times [0,t_0])$. By the above Lemmas it follows that for every $i\in [0,...,\frac{H_n}{t_0}]$ and every $x,y \in T_{t_0i}(V_n)$, 
$$d_{\epsilon^2H_n}(x,y)<\epsilon^2.$$
It remains to notice that the collection  $\left(T_{t_0i}(V_n)\right)_{i=0}^{\frac{H_n}{t_0}}$ covers $\beta-\epsilon$ of space.
\end{proof}

So it remains to proof Lemmas \ref{mesflow}.

\begin{proof}[Proof of Lemma \ref{mesflow}]
We will use a special representation for $(T_t)_{t\in \R}$ (see Section \ref{def.rank}). Fix $\epsilon>0$. Since $(T_t)_{t\in\R}$ is measurable, there exists $t_0=t_0(\epsilon)$ such that for every $t\leq t_0$, we have
\begin{equation}\label{equ1}
\epsilon^4> \sum_{i=1}^k\mu(T_tP_i\triangle P_i). 
\end{equation}

Define
$$
g_t(y):=\sum_{i=1}^k\chi_{T_tP_i\triangle P_i}(y).
$$
By \eqref{equ1} we get that for every $t\leq t_0$,

$$\int_{Y^f}g_t(y) d\nu^f<\epsilon^4.$$ 
Therefore,
$$
\int_0^{H_n}\left(\int_{B_n}\left(\int_0^{t_0}g_{t}(T^sx)dt\right)d\nu\right)ds\leq \int_{Y^f} \left(\int_0^{t_0}g_{t}(T^sx)dt\right)d\nu^f<\epsilon^4t_0.
$$
By Fubini's theorem
$$
\int_{B_n\times[0,t_0]}\left(\int_0^{H_n}g_{t}(T^sx)ds\right)d\nu\times dt<\epsilon^4t_0.
$$
Hence there exists a set $W_{n,t_0}\subset B_n\times[0,t_0]$, $\nu\times\lambda(W_{n,t_0})\geq (1-\epsilon)t_0\nu(B_n)$, such that for every $(y,t)\in W_{n,t_0}$
$$
\int_0^{H_n}g_{t}(T^sy)ds<\epsilon^2 H_n.
$$ 
So for $(y,t)\in W_{n,t_0}$
$$
d_{H_n}(y,T_ty)<\epsilon.
$$
This finishes the proof.
\end{proof}

\section{Slow entropy of Arnol'd flows, proof of Theorem \ref{main}}\label{Arn}
In this section we will prove Theorem \ref{main}. Recall that $\a\in \mathcal{E}\cap \mathcal{D}$ and $f$ has asymmetric logarithmic singularities. For simplicity assume that $\int_\T f d\lambda=1$ and $A_i=1$, $B_i=2$ for $i=1,2,3$ (in \eqref{asu}, \eqref{asu2}, \eqref{asu3}). For fixed $m\in\N$ consider the following partition $\cP_m$ of $\T^f$: 
The set $\{(x,s)\;:\; f(x)\geq \log m\}$ is one atom of the partition. Now, partition 
\begin{equation}\label{km0}K_m:=\{(x,s)\;:\; f(x)<\log m\}
\end{equation} 
into squares of diameter between  $\frac{1}{m}$ and $\frac{2}{m}$. Note that the sequence of partitions $(\cP_m)$ is generating. Therefore, in view of Proposition \ref{calc}, Theorem \ref{main} follows by the three propositions below (see \eqref{alm} and \eqref{epdel} for definitions):

\begin{pr}\label{nr1} Fix $\beta\in (0,1]$. There exists a sequence $k_n\to +\infty$ such that for every $m\in\N$, every sufficiently small $\epsilon>0$ and every $t>1$, 
\begin{equation}\label{lem.1}
\lim_{n\to+\infty}
\frac{S^{k_n}_{\cP_m}(\epsilon,\beta)}{k_n(\log k_n)^t}=0.\end{equation}
\end{pr}
Notice that Proposition \ref{nr1} gives the upper bound for $h_s^\beta$. The lower bound is given by the two following Propositions:

\begin{pr}\label{sea}If $(T_t^f)$ satisfies the PD-property, then for  every $\beta\in (0,1]$ there exists $m_0\in \N$ such that for every $m\geq m_0$, every (sufficiently small) $\epsilon>0$ and every $t<1$  
\begin{equation}\label{low.b}\liminf_{r\to+\infty}
\frac{S^{r}_{\cP_m}(\epsilon,\beta)}{r(\log r)^t}=+\infty.
\end{equation}
\end{pr}

\begin{pr}\label{nr2} $(T_t^f)$ satisfies PD-property.
\end{pr}

We will prove each of the above propositions in a separate subsection. The proof of Proposition \ref{nr1} will be conducted simultanously with the proof of Proposition \ref{nr3} (the proofs follow the same lines) in Subsection \ref{pr3}. 

\subsection{Lower bound on orbit growth, proof of Proposition \ref{sea}}
For the proof of Proposition \ref{sea} we need the following lemma (let $\omega_n=\log(\log(n))$):

\begin{lm}\label{mim2} There exists $k_0\in\N$ such that for every $k\geq k_0$ and every $x,y\in \T$ satisfying 
\begin{equation}\label{dxy2}
\frac{\omega^2_{k+1}}{(k+1)\log (k+1)}\leq \|x-y\|<\frac{\omega^2_k}{k\log k}, 
\end{equation}
we have $|J_{x,y}|<k\omega_k^3$.
\end{lm}
\begin{proof} Recall that $0\in J_{x,y}=[A_{x,y},B_{x,y}]$. We will show that 
$$
|A_{x,y}|,|B_{x,y}|<\frac{1}{2}k\omega_k^3.
$$
We will show the above inequality for $B_{x,y}$ (the proof for $A_{x,y}$ follows the same lines).
Let $n\in \N$ be unique such that $q_n\leq k<q_{n+1}$. Let 

$$I^n_{bad}:=\left[-\frac{2\omega_{q_n}}{q_n\log q_n},\frac{2\omega_{q_n}}{q_n\log q_n}\right].$$

Let $R_{n,k}:=\frac{1}{2}\min(k\omega_k^3,q_{n+1})$. We claim that there exists an interval $[M_1,M_2]\subset [0,R_{n,k}]$ such that $|M_2-M_1|\geq \max(\frac{1}{8}R_{k,n},q_n)$ and
\begin{equation}\label{costam}
\bigcup_{M_1}^{M_2}T^i[x,y]\cap I^n_{bad}=\emptyset.
\end{equation}
Indeed, notice that for $0\leq i<\frac{q_{n+1}}{2q_n}$, we have
$$
T^{iq_n}(I^n_{bad})=I^n_{bad}+i (q_n\a \; \text{mod} \;1).
$$
Therefore for $i\geq \frac{4q_{n+1}\omega_{q_n}}{q_n\log q_n}$ 
$$
T^{iq_n}(I^n_{bad})\cap I^n_{bad}=\emptyset.
$$
Obviously for $0<j<\frac{q_{n+1}}{2}$ not divisible by $q_n$ we have 
\begin{equation}\label{disj}
T^{j}(I^n_{bad})\cap I^n_{bad}=\emptyset.
\end{equation}
It remains to notice that by diophantine assumptions on $\alpha$ we have 
$$\frac{R_{n,k}}{16q_n}\geq \frac{4q_{n+1}\omega_{q_n}}{q_n\log q_n}$$
 to get that for every $j\in [\frac{R_{n,k}}{16},R_{n,k}]$, \eqref{disj} holds. This gives \eqref{costam}. 

By \eqref{costam}, it follows that \eqref{empt2} holds for $R=M_2-M_1$ and $x+M_1\alpha, y+M_1\alpha$ ($\kappa_n=\omega_{q_n}$ modulo additive constant). Therefore, by \eqref{birk.cont2} for $r=M_2-M_1$
\begin{multline*}
|f^{(M_2-M_1)}(x+M_1\a)-f^{(M_2-M_1)}(y+M_1\a)|\geq \\
\frac{9}{10}|M_2-M_1|\log(M_2-M_1)\|x-y\|\geq \frac{1}{64}R_{n,k}\frac{\omega^2_k}{k\log k}\geq \omega_k.
\end{multline*}
By cocycle identity, we have 
$$
\max\left( |f^{(M_1)}(x)-f^{(M_1)}(y)|,|f^{(M_2)}(x)-f^{(M_2)}(y)\right)>1.
$$
Hence $B_{x,y}\leq M_2<R_{n,k}$ and this finishes the proof.
\end{proof}

Using the above lemma, we can prove Proposition \ref{sea}:

\begin{proof}[Proof of Proposition \ref{sea}] 
Fix $\beta\in (0,1]$ and let $c_0,c_1$ come from the PD-property.
We will use Proposition \ref{db} for the sequence of partitions $(\cP_m)_{m\in\N}$. Define $\eta= 10^{-6}\min(\beta^{10},c_0^{10},c_1^{10})$ and let $m_0=m_0(\eta^4), R_0=R_0(\eta^4),V(\eta^4)$ come from Proposition \ref{db}. Let $k_0$ come from Lemma \ref{mim2}. By enlarging $m_0$ we can assume that $m_0\gg \max(\eta^{-1},k_0)$. Fix any  $\epsilon<m_0^{-3}$ and $m\geq m_0$. We will show that for $r$ sufficiently large, we have
\begin{equation}\label{boun}
S^{r}_{\cP_m}(\epsilon,\beta)\geq \frac{r\log r}{\omega_r^{20}}.
\end{equation}
This will finish the proof of \eqref{low.b}. By Egorov's theorem, there exists a set $U$, $\mu(U)>1-\eta^4$ and $M_0=M_0(\eta)$ such that for every $x\in U$ and $M\geq M_0$
\begin{equation}\label{eg1}
\frac{1}{2}M<f^{(M)}(x)<2M.
\end{equation}
Define $G^f=U^f\cap K_{m_0}$ (see \eqref{km0})

 Since $m_0\gg \eta^{-1}$ it follows that $\mu(G^f)\geq 1-\eta^3$ (see \eqref{km0}). By Egorov's theorem for $\chi_{G^f}$, there exists a set $W\subset \T^f$,  $\mu(W)\geq 1-3\eta$ and $R_1=R_1(\eta)\in \N$ such that for every $(x,s)\in W$ and $R\geq R_1$ we have
\begin{equation}\label{birksi}
\left|\left\{t\in [0,R]\;:\; T_t^f(x,s)\in G^f\right\}\right|\geq (1-4\eta)R.
\end{equation}
 Fix $r\gg \max(R_0,R_1,m_0)$ and define 
\begin{equation}\label{cyr}
C_r:=\T\setminus\bigcup_{i=0}^{\frac{2r}{\inf_\T f} }\left[-\frac{1}{r\omega_r}-i\alpha,
\frac{1}{r\omega_r}-i\alpha\right].
\end{equation}

Notice that since $r\gg m_0\gg \eta$, we have $\mu(C_r^f)>1-\frac{1}{\omega_r\inf_\T f}\geq 1-\eta$ and therefore  $\mu(G^f\cap V\cap W\cap C_r^f)\geq 1-10\eta$. Take any $(x,s),(y,s')\in G^f\cap V\cap W\cap C^f_r$  such that 
\begin{equation}\label{sba}
(y,s')\in B^{\cP_m}_r((x,s),\epsilon). 
\end{equation}
This means that
\begin{equation}\label{hamd}d^{\cP_m}_r((x,s),(y,s'))<\epsilon< \min(c_0^2,c_1^2).\end{equation} 

So by Proposition \ref{db}, there exists an interval $[A,B]\subset [0,r]$, $B-A>\frac{c_1r}{10}$, such that for every $t\in [A,B]$ we have
$$
d^f\left(T^f_t(x,s),T^f_t(y,s')\right)<10^{-2}.
$$

By \eqref{birksi}, \eqref{hamd} and the fact that  $B-A>\frac{c_1r}{10}$ it follows that there exists $t_0\in [A,B]$ such that  
\begin{equation}\label{compl}T^f_{t_0}(x,s),T^f_{t_0}(y,s')\in G^f
\end{equation}
 and $T^f_{t_0}(x,s),T^f_{t_0}(y,s')$ are in one atom of $\cP_m$. 
By the definition of $G^f$ and the partition $\cP_m$ this implies in particular that ($m_0\gg k_0$)
\begin{equation}\label{cjmp}\rho^f(T^f_{t_0}(x,s),T^f_{t_0}(y,s'))<
\frac{2}{m_0}\leq \frac{1}{k_0\log k_0}.\end{equation}

 Let $r_1, r_2\in \N$ be such that the first coordinates of $T^f_{t_0}(x,s),T^f_{t_0}(y,s')$ are respectively $x+r_1\alpha$ and $y+r_2\alpha$. Notice that since $t_0\leq r$, we have $r_1,r_2\leq r\inf_\T f$. Moreover by \eqref{compl}, we get that 
\begin{equation}\label{eq12}
x+r_1\alpha,y+r_2\alpha\in U
\end{equation}
 and by \eqref{cjmp}, we have $\|x+r_1\alpha-(y+r_2\alpha)\|<\frac{1}{k_0\log k_0}$. 

Consider the interval $J_{x+r_1\alpha,y+r_2\alpha}$ and denote the endpoints of $J_{x+r_1\alpha,y+r_2\alpha}$ by $C_{x,y}<0$ and $D_{x,y}>0$.  Using Lemma \ref{mim2},  
we get 
$$
D_{x,y}-C_{x,y}\leq k\omega_k^3,
$$
where $k$ satisfies \eqref{dxy2} for $x+r_1\alpha, y+r_2\alpha$.
By the definition of $[A,B]$ and $J_{x+r_1\alpha,y+r_2\alpha}$ (see \eqref{jxy}) it follows that 
$$
[A-t_0,B-t_0] \subset [f^{(C_{x,y})}(x+r_1\alpha),f^{(D_{x,y})}(x+r_1\alpha)].
$$
Therefore, by \eqref{eq12} and \eqref{eg1}, we have
\begin{multline}\label{est.in}
\frac{c_1r}{10}\leq B-A\leq 2\max(|A-t_0|,|B-t_0|)\leq\\
2\max\left(\left|f^{(C_{x,y})}(x+r_1\alpha)\right|,
\left|f^{(D_{x,y})}(x+r_1\alpha)\right|\right)
\leq 4\max (|C_{x,y}|,|D_{x,y}|)<\\
4(D_{x,y}-C_{x,y})\leq 4k\omega_k^3
\end{multline}
This  by \eqref{dxy2} (and $r\gg c_1^{-1}$) implies that 
\begin{equation}\label{fgh}
\|x-y-(r_2-r_1)\alpha\|\leq \frac{\omega_k^2}{k\log k}\leq \frac{\omega_r^7}{r\log r}.
\end{equation}
Notice, that since $(x,s), (y,s')\in G^f\subset K_{m_0}$ we get $|s-s'|\leq m_0-1$. Therefore and by \eqref{cjmp}, we get
\begin{multline}\label{bsum}
|f^{(r_1)}(x)-f^{(r_2)}(y)|\leq \|(x+r_1\alpha)-(y+r_2\alpha)\|+\\
|f^{(r_1)}(x)-f^{(r_2)}(y)+s-s'|+m_0-1= 
\rho^f(T^f_{t_0}(x,s),T^f_{t_0}(y,r))+m_0-1\leq m_0.
\end{multline}
Moreover since $r_1<r\inf_\T f$, $x+r_1\alpha,y+r_2\alpha\in C_r$ (see \eqref{compl}) and \eqref{fgh} holds, we get that the assumptions of Corollary \ref{wniosek2} are satisfied. Therefore, using \eqref{fgh}, we get   
\begin{equation}\label{bsum2}
|f^{(-r_1)}(x+r_1\alpha)-f^{(-r_1)}(y+r_2\alpha)|\leq \omega_r^{10}.
\end{equation}
Using \eqref{bsum} and \eqref{bsum2}, we get

\begin{multline*}
m_0\geq |f^{(r_1)}(x)-f^{(r_2)}(y)|=|f^{(-r_1)}(x+r_1\alpha)-
f^{(-r_1)}(y+r_2\alpha)+\\
f^{(r_1-r_2)}(y+r_2\alpha)|\geq 
|f^{(r_1-r_2)}(y+r_2\alpha)|+ \omega_r^{10},
\end{multline*} 
which implies that $|r_1-r_2|\leq 2\omega_r^{11}$ ($r\gg m_0$). By this and \eqref{fgh} we get that if $(y,s')$ and $(x,s)$ satisfy \eqref{sba}, then 
$$
(y,s')\in \left(\bigcup_{i=- 2\omega_r^{11}}^{ 2\omega_r^{11}}
\left[x+i\a-\frac{\omega_r^7}{r\log r},x+i\a-\frac{\omega_r^7}{r\log r}\right]\right)^f.
$$
Therefore, for every $(x,s)\in G^f\cap V\cap W\cap C_r^f$
$$
\mu(B^{\cP_m}_r((x,s),\epsilon)\cap G^f\cap V\cap W\cap C^f_r)\leq \frac{\omega_r^{20}}{r\log r}.
$$
Now since $r\gg m$, and $\mu(G^f\cap V\cap W\cap C^f_r)\geq 1-10\eta \gg 1-\beta$ this implies that \eqref{boun} holds. This finishes the proof of Proposition \ref{sea}.
\end{proof}

\subsection{PD-property for Arnol'd flows, proof of Proposition \ref{nr2}}

To prove Proposition \ref{nr2} we will show that Proposition \ref{cocy} holds. In proving that Proposition \ref{cocy} holds, we will be using the full strength of the diophantine condition on $\alpha\in \mathcal{D}\cap \mathcal{E}$.

Recall that $(\kappa_n)$ is going slowly to $+\infty$ (i.e. $\kappa_n=\log n$).

\begin{proof}[Proof of Proposition \ref{nr2}]

Fix $\epsilon>0$. Let 
$$
Z_n:=\bigcup_{i=-q_n}^{q_n}T^i
\left[-\frac{1}{q_n\log^{7/8}q_n},\frac{1}{q_n\log^{7/8}q_n}\right]
$$
and (for $k\geq 2$)
\begin{equation}\label{defz1}
Z_{k}:=\bigcup_{n\geq k, n\notin K_\alpha}Z^c_n.
\end{equation}
Notice that $\lambda(Z_{k})\geq 1-
2\sum_{n\geq k,n\notin K_\alpha}\frac{1}{\log^{7/8}q_n}$. 
Define $Z_1:=Z_{k_0}$ where   $k_0$ is such that $\lambda(Z_{k_0})\geq 1-\epsilon^2$ (such $k_0$ exists since $\alpha\in \mathcal{E}$). Let $\delta''\ll q_{k_0}^{-1}$ and take any $x,y\in Z_1$ such that $\|x-y\|<\delta''$.
Let $n\geq k_0$ be such that 
\begin{equation}\label{distxy}
\frac{1}{q_{n+1}\log q_{n+1}}<\|x-y\|\leq \frac{1}{q_{n}\log q_{n}} 
\end{equation} 
and $M\in [q_n,q_{n+1})\cap \N$ be unique such that 
\begin{equation}\label{distxy2}
\frac{1}{(M+1)\log (M+1)}<\|x-y\|\leq \frac{1}{M\log M}. 
\end{equation}
Let moreover $i\in \N$ be the smallest number such that $q_{n+i}>10^4q_n$. Denote the endpoints of $J_{x,y}$ by $C_{x,y}\leq 0$ and $D_{x,y}\geq 0$. We consider two cases:

\textbf{Case 1.} $n+i\in K_\alpha$. Notice that the orbit of length $q_{n+i}$ of any point $z_0\in\T$ is at least $\frac{1}{2q_{n+i}}\geq \frac{1}{2q_{n+i-1}\log q^{7/8}_{n+i-1}}$ spaced. Therefore there is at most one 
$i_0\in \left(-\frac{q_{n+i}}{2},\frac{q_{n+i}}{2}\right)\cap \Z$ such that 
\begin{equation}\label{badpoint}
x+i_0\alpha\in [-\frac{1}{4q_{n+i-1}\log^{7/8}q_{n+i-1}},\frac{1}{4q_{n+i-1}\log^{7/8} q_{n+i-1}}].
\end{equation}
Assume WLOG that $i_0<0$. Then by the definition of $i$, \eqref{empt2} is satisfied for $R=\frac{q_{n+i}}{2}$, $n\in \N$ and $x,y\in \T$ (with $\kappa_n=\log n$). Define $K_{x,y}:=[\frac{M}{200},\frac{M}{100}]\cap \Z$. By \eqref{birk.cont2} and \eqref{distxy2}, for every $r\in K_{x,y}$, ($r\geq M/200\geq q_n/200$) we get
$$
|f^{(r)}(x)-f^{(r)}(y)|\geq  \frac{9}{2000}M\log M\|x-y\|\geq 1/400. 
$$
Moreover by \eqref{birk.cont}, \eqref{birk.cont2} and \eqref{distxy2},
 for every $r\in [0, \frac{M}{100}]$, we get 
$$
|f^{(r)}(x)-f^{(r)}(y)|\leq \frac{11}{1000} M\log M\|x-y\| <1/80. 
$$
Finally, by \eqref{birk.cont2} for $\tilde{M}=[M/2]<\frac{q_{n+i}}{2}$ and \eqref{distxy2},
 we get
$$
|f^{(\tilde{M})}(x)-f^{(\tilde{M})}(y)|\geq  \frac{8}{20} M\log M \|x-y\|\geq 1/10. 
$$
By the three equations above, we get $|K_{x,y}|\geq \frac{M}{200}$, $K_{x,y}\subset J_{x,y}$ and $D_{x,y}<M/2$. We will show that 
\begin{equation}\label{cxy}
|C_{x,y}|<M/2.
\end{equation}
This will finish the proof since then $|K_{x,y}|\geq\frac{|J_{x,y}|}{200}$. Therefore it remains to show \eqref{cxy}. For this aim we will show that there exists $-M/2<n_0<0$ such that $|f^{(n_0)}(x)-f^{(n_0)}(y)|\geq 1/12$ (then $|C_{x,y}|\leq|n_0|<M/2$). Let $-\frac{q_{n+i}}{2}<i_0<0$ satisfy \eqref{badpoint} (if such $i_0$ does not exist, let $i_0=0$). Consider the longer of the intervals $[-M/2,i_0-1]$ and $[i_0+1,0]$ and call it $I_M=[A,B]$. Then $|I_M|\geq M/4-1$. Moreover by the definition of $i_0$ it follows that (for $\kappa_n=\log n$), we have
$$
\bigcup_{i=0}^{A-B}T^{i}[x+B\alpha,y+B\alpha]\cap
\left[-\frac{\kappa_n}{q_n\log q_n},\frac{\kappa_n}{q_n\log q_n}\right]=\emptyset.
$$
So \eqref{empt2} is satisfied with $R=A-B$ and $x+B\alpha,y+B\alpha\in \T$. 
Using \eqref{birk.cont2} with $r=A-B$, we get 
\begin{multline*}
|f^{(A-B)}(x+B\alpha)-f^{(A-B)}(y+B\alpha)|\geq\\ \frac{9}{10}|A-B|\log|A-B|\frac{1}{M+1\log(M+1)}\geq 1/6.
\end{multline*}
Using cocycle identity, we have
$$
\max\left(|f^{(A)}(x)-f^{(A)}(y)|,|f^{(B)}(x)-f^{(B)}(y)|\right)\geq 1/12.
$$
We conclude by setting $n_0$ to be either $A$ or $B$ depending on which number above obtains the maximum. This finishes the proof of \textbf{Case 1.}

\textbf{Case 2.} $n+i\notin K_\alpha$. Let $W_n:=\frac{q_{n+i}}{3\log^{7/8}q_{n+i-1}}$. Since $x,y\in Z_1$, \eqref{distxy2} is satisfied,  $n+i\gg k_0$ and $q_{n+i-1}<10^4q_n$, we get 
\begin{equation}\label{mn}
\bigcup_{i=-W_n}^{W_n}T^{i}[x,y]\cap\left[-\frac{\kappa_n}{q_n\log q_n},\frac{\kappa_n}{q_n\log q_n}\right]=\emptyset.
\end{equation}
We will consider two cases:

\textbf{A.} $\|x-y\|\geq \frac{1}{30W_n\log W_n}$. Then, by \eqref{distxy2}, we have $W_n \geq M/30$. Define $K_{x,y}:=[\frac{M}{200},\frac{M}{100}]\cap \Z$. 

By \eqref{mn}, we get that \eqref{empt2} is satisfied for $M/100\leq W_n$. So by \eqref{birk.cont2}, for every $r\in K_{x,y}$, we have
$$
|f^{(r)}(x)-f^{(r)}(y)|\geq  \frac{1}{250}M\log M\|x-y\|\geq \frac{1}{400}. 
$$
Moreover, by \eqref{birk.cont} and \eqref{birk.cont2}, for every $r\in [0,\frac{M}{100}]$,  we have 
$$
|f^{(r)}(x)-f^{(r)}(y)|\leq \frac{11}{1000} M\log M\|x-y\| <1/80. 
$$
Finally, by \eqref{birk.cont2} for $\tilde{M}\in \{-[\frac{M}{30}],[\frac{M}{30}]\}$ (note that $|M|/30\leq W_n$)
$$
|f^{(\tilde{M})}(x)-f^{(\tilde{M})}(y)|\geq  \frac{9}{300} M\log M \|x-y\|\geq 1/40. 
$$
Puting this together, we get $|K_{x,y}|\geq \frac{M}{200}$, $K_{x,y}\subset J_{x,y}$ and $|J_{x,y}|\leq |C_{x,y}|+|D_{x,y}|\leq M/30+M/30=M/15$. This finishes the proof in case \textbf{A.}

\textbf{B.} $\|x-y\|< \frac{1}{30W_n\log W_n}$ (then $M\geq 30W_n$). 

Since $x,y\in Z$, it follows that one of the following holds:
 \begin{equation}\label{fc}
\bigcup_{i=0}^{[\frac{q_{n+1}}{2}]}T^{i}[x,y]\cap\left[-\frac{\kappa_n}{q_n\log q_n},\frac{\kappa_n}{q_n\log q_n}\right]=\emptyset
\end{equation}
 or
\begin{equation}\label{sc}
\bigcup_{i=0}^{-[\frac{q_{n+1}}{2}]}T^{i}[x,y]\cap\left[-\frac{\kappa_n}{q_n\log q_n},\frac{\kappa_n}{q_n\log q_n}\right]=\emptyset.
\end{equation}
Indeed, whether \eqref{fc} of \eqref{sc} holds depends on the sign of $\alpha-\frac{p_n}{q_n}$ and the position of the point closest to $0$ of the orbit of $x$ of length $q_n$. It follows that going either forward or backward we can always recede from the dangerous interval around $0$. A complete proof of this statement is given in \cite{FK} Lemma 4.6.

We will conduct the proof in case \eqref{fc} holds (the proof if \eqref{sc} holds is analogous). Define $K_{x,y}:=[\frac{M}{200},\frac{M}{100}]\cap \Z$. Notice that \eqref{fc} implies that \eqref{empt2} holds for $R=q_{n+1}/2$.
So by \eqref{birk.cont2} and \eqref{distxy2}, for every $r\in K_{x,y}$, we have 
$$
|f^{(r)}(x)-f^{(r)}(y)|\geq  \frac{1}{250}M\log M\|x-y\|\geq 1/400. 
$$
Moreover, by \eqref{birk.cont} and \eqref{birk.cont2}, for every $r\in [0,\frac{M}{100}]$, we have 
$$
|f^{(n)}(x)-f^{(n)}(y)|\leq \frac{11}{1000} M\log M\|x-y\| <1/80. 
$$
Finally, by \eqref{birk.cont} and \eqref{distxy2} for $\tilde{M}=[\frac{M}{2}]$ we have
$$
|f^{(\tilde{M})}(x)-f^{(\tilde{M})}(y)|\geq  8/20 M\log M \|x-y\|\geq 1/8. 
$$
By the three equation above, we get $|K_{x,y}|\geq \frac{M}{200}$, $K_{x,y}\subset J_{x,y}$ and $D_{x,y}<M/2$. We will show that 
\begin{equation}\label{cxy2}
|C_{x,y}|<M/2.
\end{equation}
this will finish the proof since then $K_{x,y}\geq \frac{|J_{x,y}|}{200}$. For this aim we will show that there exists $-M/2<n_0<0$ such that $|f^{(n_0)}(x)-f^{(n_0)}(y)|\geq 1/16$. Consider intervals $I_1=[-M/6,0], I_2=[-M/3,-M/6], I_3=[-M/2,-M/3]$. Since $M\geq 30 W_n$, it follows that there exists $i\in\{1,2,3\}$ such that for every $r\in I_i$, we have 
$$
T^r[x,y]\cap \left[-\frac{\kappa_n}{q_n\log q_n},\frac{\kappa_n}{q_n\log q_n}\right]=\emptyset.
$$
Denote this interval by $[A,B]$. Then $|A-B|\geq M/6$ and  we have
$$
\bigcup_{i=0}^{A-B}T^{i}[x+B\alpha,y+B\alpha]\cap\left[-\frac{\kappa_n}{q_n\log q_n},\frac{\kappa_n}{q_n\log q_n}\right]=\emptyset.
$$
Hence \eqref{empt2} is satisfied for $R=A-B$ and $x+B\alpha, y+B\alpha\in\T$. By \eqref{birk.cont2} with $r=A-B$, we get 
\begin{multline*}
|f^{(A-B)}(x+B\alpha)-f^{(A-B)}(y+B\alpha)|\geq \\
9/10|A-B|\log|A-B|\frac{1}{M+1\log(M+1)}\geq 1/8.
\end{multline*}
Using cocycle identity, we have
$$
\max\left(|f^{(A)}(x)-f^{(A)}(y)|,|f^{(B)}(x)-f^{(B)}(y)|\right)\geq 1/16.
$$
We conclude by setting $n_0$ to be either $A$ or $B$ depending on which number above obtains the maximum. We get $|C_{x,y}|\leq M/2$. This finishes the proof of \textbf{B.} and hence also the proof of Proposition \ref{cocy}.
\end{proof}

\section{Slow entropy of Kochergin flows, proof of Theorem \ref{main2}}\label{Koc}
Le $(\cP_m)$ be the sequence of partitions of $\T^f$ defined in Section \ref{Arn}
Similarly to Section \ref{Arn}, Theorem \ref{main2} follows by the following two propositions (see \eqref{alm} and \eqref{epdel}):

\begin{pr}\label{nr3} Fix $\beta\in (0,1]$.
 There exists a sequence $k_n\to +\infty$ such that for any $m\in \N$ , any sufficiently small $\epsilon>0$ and any $t>1+|\gamma|$, 
$$\lim_{n\to+\infty}
\frac{S^{k_n}_{\cP_m}(\epsilon,\beta)}{k_n^t}=0.$$
\end{pr}
The above proposition gives an upper bound on $h_s^\beta(T_t)$. The more difficult part is the lower bound, which is given in the proposition below. 

\begin{pr}\label{nr4} For every $\beta\in(0,1]$ there exists $m_0\in \N$ such that for every $m\geq m_0$, every (sufficiently small) $\epsilon>0$ and every $t<1+|\gamma|$ 
$$\liminf_{r\to+\infty}
\frac{S^{r}_{\cP_m}(\epsilon,\xi)}{r^t}=+\infty.$$
\end{pr}
Theorem \ref{main2} is an obvious consequence of Propositions \ref{nr3} and \ref{nr4}. We will give the proofs of Propositions \ref{nr3} and \ref{nr4} in separate subsections.

\subsection{Upper bound on orbit growth, proofs of Propositions \ref{nr1} and \ref{nr3}}\label{pr3}
We will give the proof of Propositions \ref{nr1}. The proof of Proposition \ref{nr3} follows the same lines-- one just needs to change the scale. We will indicate modifications one has to make to prove Proposition \ref{pr3}. Before we give a strict proof, let us first give an outline.

\textbf{Outline of the proof:} Assume $f$ has asymmetric logarithmic singularities. Then for "typical" (in measureable sense) $\theta\in \T$, we have 
$$\sup_{0\leq i\leq q_n}|f'^{(i)}(\theta)|\leq C q_n\log q_n.$$
Fix $t>1$. For a "typical" pair $x,y\in \T$ such that $\|x-y\|\leq \frac{1}{q_n\log^{t}q_n}$  
$$
\sup_{0\leq i\leq q_n}|f^{(i)}(x)-f^{(i)}(y)\leq \|x-y\|\sup_{0\leq i\leq q_n}|f'^{(i)}(\theta)|\ll 1.
$$
Therefore, if $|s-s'|<\frac{1}{2m}$, then $(x,s)$ and $(y,s')$ are in one topological ball of length $q_n$-- the points do not split until $q_n$. So they are also in one Hamming ball of length $q_n$. We get
$$
S^{q_n}_{\cP_m}(\epsilon,\beta)\leq C(\epsilon, \beta) q_n\log^{t}q_n.
$$
This gives \eqref{lem.1} and hence also Proposition \ref{nr1}. Analogous argumentation (with different scale and $t>1+|\gamma|$ instead of $t>1$) gives an outline of the proof of Proposition \ref{nr3}.
\begin{uw}\label{rem.liou} Proposition \ref{nr3} (and hence also Theorem \ref{main2}) does not hold if $\alpha$ is  well aproximable by rationals, then the upper bound for the Bowen balls (and hence also Hamming balls) is different. For simplicity we will discuss the case  $\alpha$ is liouvillean, i.e. $q_{n+1}\geq e^{q_n}$ for a subsequence of denominators (although what follows is true for well aproximable diophantine $\alpha$). Consider the set $S:=\bigcup_{i=0}^{-q_n+1}[\frac{1}{4q_n},\frac{1}{3q_n}]$. Then for every $\theta\in S$ and $N=\left[\frac{q_{n+1}}{10}\right]$, we have
$$
\sup_{0\leq i\leq N}|f'^{(N)}(\theta)\leq Cq_{n+1}q_n^\gamma.
$$
Therefore, if $x,y\in S$ satisfy $\|x-y\|\leq \epsilon^2(q_{n+1}q_n^{\gamma})^{-1}$ then they are in one Bowen ball (of length $N$). So 
$$
S^{N}_{\cP_m}(\epsilon,\beta)\leq \epsilon^{-2}q_{n+1}q_n^{\gamma}\leq N\log N,
$$
which show that the asymptotical number of balls is much smaller than $N^{1+\gamma}$. Hence one cannot use this technique for Liouvillean $\alpha$. On the other hand this should be enough to disprove the local rank one of Kochergin flow for every $\alpha$ (since the growth is superlinear for every $\alpha$).
\end{uw}
 



\begin{proof}[Proof of Proposition \ref{nr1}] Recall that we assume that $\int_\T f(x) d\lambda=1$. Notice that by the definition of $S^{k_n}_{\cP_m}(\epsilon,\beta)$, it is enough to prove \eqref{lem.1} for $\beta=1$. Let $k_n=q_n$ for $n\in \N$.
Fix $t>1$, $t_0\in (1,t)$ ($1+|\gamma|<t_0<t$ for Proposition \ref{nr3}), $m\in \N$ and $\epsilon\ll m^{-1}$.

Take $\xi=\xi(\epsilon)>0$ such that the $\xi$-neighbourhood $V_\xi$ of the boundary of $\cP_m$ has measure $<\frac{1}{10}\epsilon$. \\

By Birkhoff theorem (using Egorov's theorem for $\chi_{V_\xi}$), there exists a set $X_\epsilon\subset X^f$, $\mu^f(X_\epsilon)>1-\frac{1}{2}\epsilon$ and a number $k_0$ such that for all $M>k_0$ and $(x,s)\in X_\epsilon$ we have

\begin{equation}\label{birk}
\frac{1}{M}\int_0^M\chi_{V_\xi}(T_t^f(x,s))dt<\frac{1}{3}\epsilon,
\end{equation}
and
\begin{equation}\label{birk2}
\left|\frac{1}{M}f^{(M)}(x)-\int_\T f d\lambda\right|<\frac{1}{3}\epsilon.
\end{equation}
Moreover, since $t_0>1$, we have, for all $n$ sufficiently large
\begin{equation}\label{delt}(\log q_n)^{t_0-1}\geq \xi^{-10}
\end{equation}
(\eqref{delt} becomes $q_n^{t_0-1-|\gamma|}\geq \xi^{-10}$ in Proposition \ref{nr3}).

For such $n$, consider the set (with $\kappa_n=\log n$)
$$V_n:=\left(\bigcup_{i=0}^{q_n-1}T^{-i}
\left[-\frac{2\kappa_n}{q_n\log q_n},\frac{2\kappa_n}{q_n\log q_n}\right]\right)^f\bigcup\left\{(x,s)\;:\; f(x)>\frac{1}{\epsilon^2}\right\}\bigcup X_{\epsilon}^c.$$

We have $\mu^f(V_n)<\frac{1}{2}\epsilon$. Consider the partition of the set $X^f\setminus V_k$ into rectangles $U(x,y,s)$ of the form $[x,y]\times[s,s+\eta]$, where $\eta\in [\xi^2,2\xi^2]$ (if we are "close" to the graph of $f$, we may have a degenerated rectangle) and $x,y\in \T$ satisfy
\begin{equation}\label{xz}\frac{1}{2q_n(\log q_n)^{t_0}}<|x-y|<\frac{1}{q_n(\log q_n)^{t_0}}
\end{equation}  ($\frac{1}{2q_n^{t_0}}<|x-y|<\frac{1}{q_n^{t_0}}$ for Proposition \ref{nr3}). 

 Since $\int_\T f(x)d\lambda=1$ it follows that the number of such rectangles is $\leq \frac{q_n(\log q_n)^{t_0}}{\xi^2}$ ($\leq \frac{q_n^{t_0}}{\xi^2}$ in Proposition \ref{nr3}). We will show that every such rectangle we have
\begin{equation}\label{xys}
U=U(x,y,s)\subset B^{q_n}_{\cP_m}((x,s),\epsilon). 
\end{equation}
This will finish the proof of \eqref{lem.1} since then 
$$
S^{q_n}_{\cP_m}(\epsilon)\leq \frac{q_n(\log q_n)^{t_0}}{\xi^2},
$$
($S^{q_n}_{\cP_m}(\epsilon)\leq \frac{q_n^{t_0}}{\xi^2}$ for Proposition \ref{nr3}) and since $\xi$ depends on $\epsilon$ only,  \eqref{lem.1} follows.

Let us show \eqref{xys}. Note that since $(x,s)\in V_n^c\subset X_\epsilon$, \eqref{birk2} holds for $x$ and $M=q_n$. Therefore, 
$$
B_{f^{(q_n)}(x)}((x,s),\frac{1}{2}\epsilon)\subset B_{q_n}((x,s),\epsilon)
$$
and it is enough to show that $U\subset B_{f^{(q_n)}(x)}(x,\frac{1}{2}\epsilon)$. Take $(z,r)\in U$. We will first show that for every $t\in[0,f^{(q_n)}(x)]$ we have 
\begin{equation}\label{dist}
d^f(T^f_t(x,s),T_t^f(z,r))<\xi.
\end{equation}
 By the definition of $U$ it follows that $d^f((x,s),(z,r))<2\xi^2$. We will show that for every $i=0,...q_n-1$, 
$$
|f^{(i)}(x)-f^{(i)}(z)|<\xi^2.
$$
This, by the definition of $d^f$ will finish the proof of \eqref{dist}.
Note that since $(x,s),(z,r)\in V_n^c$, \eqref{empt2} is satisfied for $R=q_n$. Therefore by \eqref{birk.cont}, \eqref{birk.cont2} and \eqref{xz}, we get

$$
|f^{(i)}(x)-f^{(i)}(z)|=100\|x-z\|q_n\log q_n\leq \xi^2,
$$
where the last inequality follows by \eqref{delt}. So \eqref{dist} follows (analogous computations can be made in Proposition \ref{nr3}). It follows by \eqref{dist} that if for some $t\in[0,f^{(q_n)}(x)]$, $T^f_t((x,s))$ and $T^f_t((z,r))$ are not in the same atom of $\cP_m$, then $T_t(x,s)\in V_\xi$ ($\eta$ neighbourhood of the boundary of $\cP_m$). Since $(x,s)\in V_n^c$, \eqref{birk2} is satisfied for $(x,s)$ and \eqref{birk} is  satisfied for $(x,s)$ and $M=f^{(q_n)}(x)$. Therefore,

$$
\bar{d}_{M}((x,s),(z,r))\leq \frac{1}{M}\lambda\{t\in [0,M]: T_tx\in V_\eta\}=\frac{1}{M}\int_{X^f}\chi_{V_\eta}(T^f_t(x,s))dt\leq \frac{1}{3}\epsilon.
$$
So \eqref{xys} holds and the proof of Proposition \ref{nr1} is finished. The proof of Proposition \ref{nr3} follows along the same lines (after the indicated changes).
\end{proof}

\subsection{Lower bound on orbit growth, proof of Proposition \ref{nr4}}\label{pr4}
The proof of Proposition \ref{nr4} uses some ideas from \cite{KRV}. For the sake of completness, we will present a self-contained proof. For simplicity, we will change notation: we will denote points in $\T^f$ simply by $x$ (not $(x,s)$). Also $d_1$ and $d_2$ denote the (pseudo) metrics on the first and second coordinate respectively, i.e. $d^f(x,y)=d_1(x,y)+d_2(x,y)$. 

Le $(\cP_m)$ be the sequence of partitions introduced in Section \ref{Arn}.
Proposition \ref{nr4} follows by the following proposition: 

\begin{pr}\label{nr5}For every $\delta>0$ there exists a set $A=A_\delta\subset \T^f$, $\mu(A)>1-\delta$ and $m_\delta,R_\delta\in \N$ such that for every $x,y\in A$, $m\geq m_\delta$, $R\geq R_\delta$ if $\bar{d}_R^{\cP_m}(x,y)<1/100$ then there exists $t_0=t_0(x,y)\in[0,R]$ such that (see \eqref{km0})
$$
T^f_{t_0}x,T^f_{t_0}y\in K_m\text{ and }d_1(T^f_{t_0}x,T^f_{t_0}y)\leq \frac{\log^{20}R}{R^{1+|\gamma|}}.
$$
\end{pr}

Before we prove Proposition \ref{nr5} let us show how it implies Proposition \ref{nr4}.

\begin{proof}[Proof of Proposition \ref{nr4}]
Fix $\delta\ll \beta$. For $R\geq 1$ define 
\begin{equation}\label{cyr2}
C_R=\bigcup_{i=0}^R\left[-\frac{1}{R\log R}+i\alpha,
\frac{1}{R\log R}+i\alpha\right].
\end{equation}
Let $R_1$ be such that $\lambda(C_{R_1})\ll \beta$. Fix  $m\geq m_\delta$ and $R\gg \max(R_\delta,R_1,m)$. Let
$$
 C:=A\cap (C^f_R)^c\cap\left\{x\in \T^f \;:\: f(x)<\log m\right\}.
$$
Notice that $\mu(C)\gg 1-\beta$. Take any $x,y\in C$  such that 
\begin{equation}\label{sba2}
y\in B^{\cP_m}_R(x,1/100). 
\end{equation}
This means that
$$d^{\cP_m}_R(x,y)\leq 1/100.$$ 
By Proposition \ref{nr5}, there exists $t_0\in [0,R]$ such that $d_1(T^f_{t_0}x,T^f_{t_0}y)\leq \frac{\log^{20}R}{R^{1+|\gamma|}}$. Let $r_1,r_2\geq 0$ be such that the first coordinates of $T^f_{t_0}x$ and $T^f_{t_0}y$ are respectively $x+r_1\alpha$ and $y+r_2\alpha$. Then \begin{equation}\label{fgh2}
\|x-y-(r_2-r_1)\alpha\|\leq \frac{\log^{20}R}{R^{1+|\gamma|}}.
\end{equation}
Since $x,y\in C\subset \left(\{x\in \T^f \;:\: f(x)<\log m\}\right)$ and $T^f_{t_0}x,T^f_{t_0}y\in K_m$ we have 
$$
2m\geq |f^{(r_1)}(x)-f^{(r_2)}(y)|.
$$
Moreover, since $x,y\notin C_R^f$ by \eqref{fgh2}, for some $\theta\in [x+r_1\alpha,y+r_2\alpha]$ we get 
$$
|f^{(-r_1)}(x+r_1\alpha)-
f^{(-r_1)}(y+r_2\alpha)|=|f'^{(-r_1)}(\theta)|
\|(x+r_1\alpha)-(y+r_2\alpha)\|.
$$
By the fact that $x\notin C_R^f$ and by \eqref{fgh2}, we get 
for $i=0,...,r_1$
$$
T^{-i}\theta\notin [-\frac{1}{2R\log R},\frac{1}{2R\log R}].
$$
So by \eqref{koks4} for $x=\theta$, $M=-r_1$ and $q_{s+1}\leq q_s\log^2q_s$, we have $|f'^{(-r_1)}(\theta)|<R^{1+|\gamma|}\log^5R$. By \eqref{fgh2}, we have
$$
|f^{(-r_1)}(x+r_1\alpha)-f^{(-r_1)}(y+r_2\alpha)|\leq \log^{25}R.
$$
Therefore, 
\begin{multline*}
2m\geq |f^{(r_1)}(x)-f^{(r_2)}(y)|=
|f^{(-r_1)}(x+r_1\alpha)-
f^{(-r_1)}(y+r_2\alpha)+\\
f^{(r_1-r_2)}(y+(r_2-r_1)\alpha)|\geq 
|f^{(r_1-r_2)}(y+(r_2-r_1)\alpha)|- \log^{25}R,
\end{multline*} 
which implies that $|r_1-r_2|\leq \log^{30}R$ ($R \gg m$). By this and \eqref{fgh2} we get that if $(y,r)$ and $(x,s)$ satisfy \eqref{sba2}, then 
$$
(y,r)\in \left(\bigcup_{i=-\log^{30}R}^{\log^{30}R}\left[x+i\a-
\frac{\log^{20}R}{R^{1+|\gamma|}},x+i\a
-\frac{\log^{20}R}{R^{1+|\gamma|}}\right]\right)^f.
$$
Therefore, for every $(x,s)\in C $
$$
\mu(B^{\cP_m}_R((x,s),1/100)\cap C)\leq \frac{\log^{51}R}{R^{1+|\gamma|}}.
$$
Since $\mu(C)\gg 1-\beta$ this implies that  
$$
S^{r}_{\cP_m}(\epsilon,\beta)\geq \frac{R^{1+|\gamma|}}{\log^{60}R}.
$$
This finishes the proof of Proposition \ref{nr4}.
\end{proof}

So it remains to prove Proposition \ref{nr5}. The reasoning in the proof of Proposition \ref{nr5} is based on \cite{KRV}.\\
\\
For $x,y\in \T^f$, $m>0$ and $R,j\in \N$ denote
\begin{multline}\label{anj}
A_j^{R,m}(x,y):=\{t\in [0,R]\;:\;d^f(T^f_tx,T^f_ty)<2m^{-1}
\text{ and }\\
2^{-j-1}<d_1(T^f_tx,T^f_ty)\leq 2^{-j}\}.
\end{multline}

We will show that Proposition \ref{nr5} follows by the following proposition:

\begin{pr}\label{techn1} For every $\delta>0$ there exists $R_\delta, m_\delta>0$ and a set $B=B_\delta\subset\T^f$, $\mu(B)>1-\delta$,  such that for every
$m\geq m_\delta$, $R\geq R_\delta$ and every $x,y\in B$ there exists a set $U_R=U_R(x,y)\subset [0,R]$ such that
\begin{enumerate}
	 \item[$(A)$] $|U_R|\geq \frac{9R}{10}$, 
    \item[$(B)$] for every $t\in U_R$, we have  $T^f_tx,T^f_ty \in K_{m_\delta}$ (see \eqref{km0}), 
    \item[$(C)$] for every $j$ such that $2^j\leq \frac{R^{1+|\gamma|}}{\log^{15}R}$ we have
$$\left|U_R\cap A_j^{R,m}(x,y)\right|\leq \frac{R}{j^2}.$$ 
\end{enumerate}
\end{pr}

Before we prove Proposition \ref{techn1}  let us show how it implies Proposition \ref{nr5}.

\begin{proof}[Proof of Proposition \ref{nr5}] Fix $\delta>0$. Take $R\geq R_\delta$, $m\geq m_\delta$,  $x,y\in B$  and let $d^{\cP_m}_R(x,y)<1/10$.
By $(B)$ in Proposition \ref{techn1}, the definition of $\cP_m$ and $U_R$, we have
\begin{multline}\label{sdf}
1/10>d^{\cP_m}_R(x,y)\geq 1-\frac{|U_R^c\cap [0,R]|}{R}-\frac{1}{R}\sum_{j\geq 0}\left|U_N\cap A_j^{R,m}(x,y)\right|\geq \\
\frac{9}{10}-\frac{1}{R}\sum_{j\geq 0}\left|U_N\cap A_j^{R,m}(x,y)\right|.
\end{multline}
Notice that by \eqref{anj} for $j\leq \frac{\log m}{2}$, we have
\begin{equation}\label{rmh}
A_j^{R,m}(x,y)=\emptyset.
\end{equation}
Let $j_R$ be such that \begin{equation}\label{dis}2^{j_R-1}\leq \frac{R^{1+|\gamma|}}{\log^{15}R}<2^{j_R}.
\end{equation}
 Then by $(C)$ in Proposition \ref{techn1} and \eqref{rmh} (since $\log m\gg 1$)
$$
\frac{1}{R}\sum_{j\leq j_R}\left|U_N\cap A_j^{R,m}(x,y)\right|\leq 1/100.
$$
Hence and by \eqref{sdf}, there exists $j_1\geq j_R$ such that 
$$
U_N\cap A_{j_1}^{R,m}(x,y)\neq\emptyset.
$$
This by the definition of $A_j^{R,m}(x,y)$ and \eqref{dis} implies that there exists $t_0\in[0,R]$ such that 
$$T^f_{t_0}x,T^f_{t_0}y\in K_{m_\delta}\subset K_m\text{ and }d_1(T^f_{t_0}x,T^f_{t_0}y)\leq \frac{\log^{20}R}{R^{1+|\gamma|}},$$
which finishes the proof of Proposition \ref{nr5}.
\end{proof}

So it remains to prove Proposition \ref{techn1}. We will do it in a separate subsection.

\subsection{Proof of Proposition \ref{techn1}}
We will use the following lemma:

\begin{lm}\label{techn} For every $\delta>0$ there exists  $R_\delta,m_\delta\in \N$ and a set $D\subset \T^f$, $\mu(D)>1-\delta$, such that for every $R\geq R_\delta$, $m\geq m_\delta$ and every $x,y\in D$ there exists a set 
$U_R=U_R(x,y)\subset [0,R]$ such that
\begin{enumerate}
	 \item $|U_R|\geq \frac{9R}{10}$, 
    \item for every $t\in U_R$, we have   $T^f_tx,T^f_ty \in K_{m_\delta}$, 
	 \item for every $w\in U_R$ such that $d(T_w^fx,T_w^fy)<2m^{-1}$, if we denote $d_1(T_w^fx,T_w^fy)=R_w^{-1}$,  then 
\begin{multline*}
|\{t\in [-R_w^{\frac{1}{1+|\gamma|}}\log^{10}R_w,
R_w^{\frac{1}{1+|\gamma|}}\log^{10}R_w]\;:
\;d_1(T^f_{t+w}(x),T^f_{t+w}(y))=R_w^{-1}\\\text{ and }
d_2(T^f_{t+w}(x),T^f_{t+w}(y))<1
\}|<R_w^{\frac{1}{1+|\gamma|}}\log^{3}R_w,
\end{multline*}
\item For $w\in U_R$ if $R_w^{-1}=d_1(T_w^fx,T_w^fy)<2m^{-1}$ then for every $t\in [-\frac{R_w}{\log^5R_w},\frac{R_w}{\log^5R_w}]$ either $d_1(T^f_{t+w}(x),T^f_{t+w}(y))=d_1(T_w^fx,T_w^fy)$ or 
$$
d_1(T^f_{t+w}(x),T^f_{t+w}(y))\geq 100d_1(T_w^fx,T_w^fy).
$$
\end{enumerate}
\end{lm}


The proof of Lemma \ref{techn} is technical, so before we prove it, let us show how it implies Proposition \ref{techn1}, and therefore also Theorem \ref{main2}.

\begin{proof}[Proof of Proposition \ref{techn1}]
Notice that $(A)$ and $(B)$ in Proposition \ref{techn1} follow by 1. and 2. in Lemma \ref{techn}. So it remains to prove $(C)$ in Proposition \ref{techn1} asuming that 3. and 4.  in Lemma \ref{techn} hold.

Fix $j$ as in $(C)$. Divide the interval $[0,R]$ into intervals $I_1,...,I_k$ of length $j^92^{\frac{j}{1+|\gamma|}}$. Since $2^j\leq \frac{R^{1+|\gamma|}}{\log^{15}R}$ it follows that $k>1$. Consider only those $I_i$, for which  $U_R\cap A_j^{R,m}(x,y)\cap I_i \neq \emptyset$. For such $i$ let $w\in U_R\cap A_j^{R,m}(x,y)\cap I_i$. By \eqref{anj} we have 
\begin{equation}\label{giho}
R_w^{-1}=d_1(T_w^fx,T_w^fx)\in[2^{-j-1},2^{-j}].
\end{equation}
 So 
\begin{equation}\label{joa}
\frac{2R_w}{\log^5R_w}\geq |I_i|.
\end{equation}
Then by 4. for $T_w^fx, T_w^fy$, \eqref{joa} and \eqref{anj}, we have 
\begin{multline*}
U_R\cap A^{R,m}_j(x,y)\cap I_i\subset\\
 \left\{t\in I_i\;:\;d(T_{t+w}^fx, T_{t+w}^fy)<2m^{-1},
d_1(T_{t+w}^fx, T_{t+w}^fy)=d_1(T_{w}^fx, T_{w}^fy)\right\}.
\end{multline*}
Moreover by 3. and \eqref{giho}, we get
\begin{multline*}
\left|\{t\in I_i\;:\;d(T_{t+w}^fx, T_{t+w}^fy)<
2m^{-1}, 
d_1(T_{t+w}^fx, T_{t+w}^fy)=
d_1(T_{w}^fx, T_{w}^fy)\}\right|\\
\leq 2^{\frac{j}{1+|\gamma|}}j^5.
\end{multline*}

Therefore, summing over all $i\in \{1,...k\}$  we get 
$$
U_R\cap A^{R,m}_j(x,y)\leq \frac{R}{2^{\frac{j}{1+|\gamma|}}j^9}2^{\frac{j}{1+|\gamma|}}j^5\leq \frac{R}{j^3}.
$$
This gives $(C)$ in Proposition \ref{techn1}.
\end{proof}

So it remains to prove Lemma \ref{techn}. Properties 1. and 2. will  follow by Birkhoff theorem type reasoning, property 3. is the most difficult and crucial, property 4. is a general fact for this dynamics. We will start by the following lemma, which also gives property 4.  In the lemma below denote $d_1(x,y)=R^{-1}_{x,y}$.
\begin{lm}\label{cons:dis} There exists $N_0\in \N$ such that for every $x,y\in \T^f$  satisfying $d^f(x,y)\leq \frac{1}{N_0}$, we have  
for  every $t\in [-\frac{R_{x,y}}{\log^4R_{x,y}},
\frac{R_{x,y}}{\log^4R_{x,y}}]$ either 
$d_1(T_t^fx,T_t^fy)=R_{x,y}^{-1}$ or 

$$
d_1(T_t^fx,T_t^fy)\geq 100R^{-1}_{x,y}.
$$
\end{lm}

\begin{proof} By diophantine assumptions on $\a$ we have
for every $m\in \Z$
\begin{equation}\label{dia}
\inf_{|s|\leq m}\|m\a\|\geq \frac{C(\a)}{m\log^2m}.
\end{equation}
Take $N_0$ such that $\log N_0\gg C(\a)$. Let $x_0, y_0\in\T$ denote the first coordinates of $x,y\in \T^f$.
For $t\in  [-\frac{R_{x,y}}{\log^4R_{x,y}},
\frac{R_{x,y}}{\log^4R_{x,y}}]$, the first coordinates of $T_t^fx$ and $T_t^fy$ are respectively $x_0+m_t\a$, $y_0+r_t\a$ for some $m_t,r_t\in \N$. Since $f>c$ it follows that $m_t,r_t<c^{-1}t$. Therefore we have either $m_t=r_t$ in which case 
$d_1(T^f_tx,T^f_ty)=\|x_0-y_0\|=d_1(x,y)$ or by \eqref{dia} and $|m_t-r_t|<2c^{-1}t\leq \frac{2R_{x,y}}{c\log^4R_{x,y}} $ we get  
$$
d_1(T^f_tx,T^f_ty)\geq \|(m_t-r_t)\a\|-|x_0-y_0|\geq \frac{c\log^2R_{x,y}}{R_{x,y}}
\geq 100R^{-1}_{x,y}.
$$
This finishes the proof.
\end{proof}

Therefore 4. in Lemma \ref{techn} follows. Notice that 4. does not depend on any other quantities and is a general fact for this dynamics. For properties 1. 2. 3. we need to define the se $D$ and $U_R$.

\subsubsection{Construction of $D$ and $U_R$ in Lemma \ref{techn}.}
Fix $\delta>0$. To simplify notation by $x_0\in \T$ we denote the first coordinate of $x\in \T^f$.
 We will construct first the set $D=D_\delta$
in  Lemma \ref{techn}.

Define first
$$
S_n:=\left\{x\in \T^f\;:\;\bigcup_{t=-q_n\log q_n}^{q_n\log q_n}T^f_t(x)\notin \left[-\frac{1}{q_n\log^3 q_n},
\frac{1}{q_n\log^3 q_n}\right]^f\right\}.
$$
Notice that $\mu(S_n)\geq 1- \frac{2}{\log^2q_n}$.
Let moreover, 
\begin{equation}\label{gse}
S:=\left(\bigcap_{n\geq n_1}S_n\right)\cap\{x\in M\;:\; x_0\notin[-n_1^{-2},n_1^{-2}]\},
\end{equation}
where $n_1=n_1(\delta)\in \N$ is such that $\mu(S)\gg 1-\delta$.

Fix a number $P_\gamma>100\gamma^{-1}$. Let for $t\in \R$,  
\begin{equation}\label{wn}
W_t:=\left\{x\in \T^f\:\: |f'^{(N(x,t))}(x_0)|\geq \frac{|N(x,t)|^{1+\gamma}}{\log^{P_\gamma}|N(x,t)|}\right\}.
\end{equation}

We have the following proposition (see Proposition in \cite{KRV}):

\begin{pr}\label{mpr} There exists $W\subset \T^f$, $\mu(W)\geq 1-\delta^{10}$ and $n_2=n_2(\delta)\in \N$ such that for every $x\in W$ and $T\geq n_2$ we have
\begin{equation}\label{eq:m}
\left|\{t\in[-T,T]\;: x\in W_{t}\;\}\right| \geq T(1-\log^{-3}T).
\end{equation}
\end{pr}

We will give the proof of Proposition \ref{mpr} in the Appendix (it follows the same lines as the proof of Proposition in \cite{KRV}). 
We have now defined all sets, which are needed in the definition of $B$ and $U_R$ in Lemma \ref{techn}. 

Consider the set 
\begin{equation}\label{defg}
G:= S \cap W\cap \{x\in \T^f\;:\; f(x_0)<\delta^{-\frac{3}{1-\gamma}}\},
\end{equation}
where $S$ is from \eqref{gse} and $W$ from Proposition \ref{mpr}. Notice that $\mu(G)\geq 1-\delta^2$. Therefore:\\
\begin{multline}
\text{ there exists a set }D=D_\delta\subset \T^f, \mu(D)\geq 1-\delta
\text{ and there exists }\\ 
n_3=n_3(\delta)\in \N\text{ such that for every }x\in D\text{, and every }R\geq n_3\text{, we have } 
\end{multline}
\begin{equation}\label{yeg2}
\left|\{t\in[0,R]\;:\;T^f_tx\in G\}\right|\geq (1-\delta)R.
\end{equation}

Define 
\begin{equation}\label{un}U_R(x,y):=\left\{t\in[0,R]
\;:\;T_t^fx,T_t^fy\in G\right\}.
\end{equation}

Notice that (for sufficiently large $R_\delta,m_\delta$) 1. and 2. in Lemma \ref{techn} are straightforward from the definition of $U_R$, $D$ and $G$. Moreover 4. follows from Lemma \ref{cons:dis}. Therefore we only need to show 3. in Lemma \ref{techn}.

\subsubsection{Proof of 3. in Lemma \ref{techn}} 
Notice that by the definition of $U_R$ (see \eqref{un}) and $G$ (see \eqref{defg}), 3. follows automatically by the following lemma. Recall that $x_0\in\T$ denotes the first coordinate of $x\in \T^f$.

\begin{lm}\label{cruc} For  every $x,y\in G$ such that $d_1(x,y)$ sufficiently small we have (for $U=d_1(x,y)^{-1}$)
\begin{multline*}
|\{t\in [-U^{\frac{1}{1+|\gamma|}}\log^{10}U,
U^{\frac{1}{1+|\gamma|}}\log^{10}U]\;:
\;d_1(T^f_{t}(x),T^f_{t}(y))=U^{-1}\\\text{ and }
d_2(T^f_{t}(x),T^f_{t}(y))<1
\}|<U^{\frac{1}{1+|\gamma|}}\log^{5}U;
\end{multline*}
\end{lm}
\begin{proof} By the definition of special flow every $t$ which belongs to the set above has to satisfy the following:
$$
|f^{(N(x,t))}(x_0)-f^{(N(x,t))}(y_0)|<2.
$$

Moreover, if $|n|>U^{\frac{1}{1+|\gamma|}}\log^{5}U$ and $x\in W_t$ (see \eqref{wn}), then by Lemma \ref{sub2},  
$$
|f^{(N(x,t))}(x_0)-f^{(N(x,t))}(y_0)|\geq 10.
$$
It remains to notice, that since $x\in G\subset W$, by Propostion \ref{mpr} we get
$$
\left|\{t\in [-U^{\frac{1}{1+|\gamma|}}\log^{10}U,
U^{\frac{1}{1+|\gamma|}}\log^{10}U]\;:\:
x\in W_t 
\right|\leq U^{\frac{1}{1+|\gamma|}}\log^{4}U.
$$
This finishes the proof.
\end{proof}

This finishes the proof of 3. in Lemma \ref{techn} and hence also the proof of Theorem \ref{main2}




\begin{lm}\label{sub2}Let $x,y\in G$ such that $d_1(x,y):=U^{-1}$ is small. \\
Then for every \\
$t \in\left[-U^{\frac{1}{1+|\gamma|}}\log^{10} U,-U^{\frac{1}{1+|\gamma|}}\log^{4}U\right]\cup \left[U^{\frac{1}{1+|\gamma|}}\log^{4}U,
U^{\frac{1}{1+|\gamma|}}\log^{10}U\right]$\\
 such that $x\in W_{t}$, we have
$$ 
|f^{(N(x,t))}(x_0)-f^{(N(x,t))}(y_0)|\geq 10.
$$
\end{lm}
\begin{proof} Let us conduct the proof for $t\geq 0$, the case $t<0$ is analogous. Recall that $x_0$ denotes the first coordinate of $x$. Let $k\in \N$ be unique such that
$$\frac{1}{q_{k+1}}\leq U^{-1}=\|x_0-y_0\|<\frac{1}{q_k}.$$
 Since $\inf_\T f>c$, we have $N(x,t)\leq c^{-1} t\leq  U^{\frac{1}{1+|\gamma|}}\log^{11} U\leq \frac{q_{k+1}}{\log^{30}q_k}$. Therefore, by diophantine assumptions on $\alpha$ and since $x\in G\subset S$ it follows that 
\begin{equation}\label{singfar}
\sup_{0\leq i <N(x,t)}d(x_0+i\a,0)\geq \frac{\log^{10}q_k}{q_k}.
\end{equation}
Therefore and since$\|x_0-y_0\|<\frac{1}{q_k}$,  
for $i=0,...,N(x,t)$,  we have
$$
-i\a\notin [x_0,y_0].
$$
So for some $\theta\in[x_0,y_0]$.
\begin{equation}\label{byh}
|f^{(N(x,t))}(x_0)-f^{(N(x,t))}(y_0)|\geq 
|f'^{(N(x,t))}(x_0)|\|x_0-y_0\|-
|f''^{(N(x,t))}(\theta)|\|x_0-y_0\|^2.
\end{equation}
But $x\in G\subset S$ satisfies \eqref{singfar}, $\theta\in [x_0,y_0]$ and $\|x_0-y_0\|\leq q_k^{-1}$. So by \eqref{koks5} and diophantine assumptions on $\alpha$ 
$$
|f''^{(N(x,t))}(\theta)|\leq N(x,t)^{2+|\gamma|}\log^3N(x,t)\leq t^{2+|\gamma|}\log^4t.
$$
Since $x\in S$, we have  $N(x,t)\geq \frac{t}{2}$. Therefore and since $x\in W_t$, we get  
$$|f'^{(N(x,t))}(x_0)|\geq \frac{N(x,t)^{1+|\gamma|}}{\log^{P_\gamma}N(x,t)}\geq \frac{t^{1+|\gamma|}}{3\log^{P_\gamma}t}.
$$
Since $t\leq U^{\frac{1}{1+|\gamma|}}\log^{10} U$, we get 
$$
|f''^{(N(x,t))}(\theta)|\|x_0-y_0\|^2\leq \|x_0-y_0\| \frac{t^{2+|\gamma|}\log^4t}{U}\leq \frac{1}{10}|f'^{(N(x,t))}(x_0)|\|x_0-y_0\|.
$$
Therefore, in \eqref{byh}, for $t\geq U^{\frac{1}{1+|\gamma|}}\log^{4} U$ we have 
$$
|f^{N(x,t)}(x_0)-f^{N(x,t)}(y_0)|\geq 
\frac{1}{2}|f'^{(N(x,t))}(y_0)||x_0-y_0|\geq \frac{t^{1+|\gamma|}}{18\log^3t}U^{-1}\geq 10.$$
This finishes the proof.
\end{proof}

\section{Appendix A.}
The proofs follow along similar lines as the proofs in \cite{KRV}.
Define $V_n=\{x\in\T^f\;:\; f'^{(n)}(x_0)\geq \frac{n^{2-\gamma}}{\log^{P_\gamma} n}\}$ (compare with the definition of $W_t$, \eqref{wn}).
Proposition \ref{mpr} will follow from the following results:

\begin{pr}\label{prob} We have 
$$
\lim_{|N|\to +\infty }\left|\frac{\log^{10\gamma^{-1}}|N|}{|N|}
\sum_{i=0}^{N-1}\chi_{V^c_i}(x)\right|=0
$$
\end{pr}

\begin{lm}\label{con} Fix $x\in \T^f$ and $M\in \Z$. If   $|f_M(y)-M|<2M^{1-\gamma}\log^{6}M$ and 

\begin{equation}\label{ase}\left|\{i\in[0,M]\cap \Z\;:\; x\notin V_i\}\right|<\frac{|M|}{\log^{10\gamma^{-1}}M},
\end{equation}
then 
\begin{equation}\label{eq:con}
\left|\{t\in[0,M]\;:\; x\notin W_{N(x,t)}\}\right|<\frac{|M|}{\log^3M}.
\end{equation}
\end{lm}

\begin{proof}[Proof of Proposition \ref{mpr}]
Fix $\delta>0$. By Proposition \ref{prob} and Egorov theorem,  there exist set $U_\delta\in \T^f$, $\mu(U_\delta)\gg 1-\delta$ and $N_0\in \N$ such that for every $x\in U_\delta$ and $|M|\geq N_0$ \eqref{ase} is satisfied.  Define $W:=U_\delta\cap S$ (see \eqref{gse}). Take $|N|$ sufficiently large and let $s$ be unique such that $q_s\leq N<q_{s+1}$
 Since $x\in S$, by \eqref{koks3} in Lemma \ref{koksi2}
$$
|f^{(N)}(y)-N|< 2N^{1-\gamma}\log^6{N}.
$$
So for any $x\in W$ and $|N|$ sufficiently large, the assumptions of Lemma \ref{con} are satisfied. Therefore \eqref{eq:con} holds for $x$ and $N$ and so also \eqref{eq:m} holds. The proof of Proposition \ref{mpr} is thus finished.  
\end{proof}

\begin{proof}[Proof of Proposition \ref{prob}]

The proof goes along the same lines as the proof of Proposition in \cite{KRV} (which in turn uses some ideas of the proof of Theorem 1. in \cite{Lyo})

\begin{lm}\label{lio}\cite{Lyo} If $Y_n$ are random variables such that $\sum_{n\geq 1} \|Y_n\|_2^2<\infty$, then $Y_n\to 0$ a.s.
\end{lm}

Let us denote $X_i:=\chi_{V^c_i}(x)$. Notice that by Lemma \ref{smd} we have   
$$
\left\|\frac{\log^{10\gamma^{-1}}|M|}
{|M|}\sum_{i=0}^{M-1}X_i\right\|_2^2\leq \log^{-30\gamma^{-1}}M.
$$
Therefore there exists a sequence $(N_k)$, $|N_{k+1}-N_k|<\frac{|N_k|}{\log^{15\gamma^{-1}} |N_k|}$ such that
$$
\sum_{k\geq 1}\left(\left\|\frac{\log^{10\gamma^{-1}}|N_k|}{|N_k|}\sum_{i=0}^{N_k-1}
X_i\right\|_2\right)^2< +\infty.
$$
By Lemma \ref{lio} we get $\frac{\log^{10\gamma^{-1}}|N_k|}{|N_k|}\sum_{i=0}^{N_k-1}X_i\to 0$ a.s. Let $k\in \N$ be unique such that $N_k\leq M<N_{k+1}$. Then
\begin{multline*}
\left|\frac{\log^{10\gamma^{-1}}|M|}{|M|}\sum_{i=0}^{M-1}X_i\right|\leq 
\left|\frac{\log^{10\gamma^{-1}}|N_k|}{|N_k|}\sum_{i=0}^{N_k-1}X_i\right|+\\
\max_{0\leq s\leq N_{k+1}-N_k}\left|\frac{\log^{10\gamma^{-1}}|N_k|}{|N_k|}
\sum_{i=N_k+1}^{N_k+s}X_i\right|.
\end{multline*}
This finishes the proof since $|N_{k+1}-N_k|<\frac{|N_k|}{\log^{15\gamma^{-1}}|N_k|}$.
\end{proof}

\begin{lm}\label{smd}
There exists a constant $C>0$ such that for every $n\in \Z$ 
$$
\mu(V^c_n)<\frac{C}{\log^{50\gamma^{-1}}n}.
$$
\end{lm}
\begin{proof} Let us conduct the proof in the case $n>0$ the case $n<0$ is analogous. Let $s\in \N$ be unique satisfying $q_s\leq n<q_{n+1}$. Let $I=(a,b]$ be any interval in the partition $\mathcal{I}$ of  $\T$ given by $\{-i\alpha\}_{i=0}^{n-1}$. It follows by \eqref{asu},\eqref{asu2}, \eqref{asu3}, that $f^{(n)}$ is $C^2$ on $I$. Moreover $\lim_{x\to b^-}f'^{(n)}(x)=+\infty$ and $\lim_{x\to a^+}f'^{(n)}(x)=-\infty$. Hence there exists $x_I\in I$ such that $f'^{(n)}(x_I)=0$. Then  for $x\in I$ there exists $\theta\in I$ such that
\begin{equation}\label{spek}
|f'^{(n)}(x)|=|f'^{(n)}(x)-f'^{(n)}(x_I)|=
|f''^{(n)}(\theta)||x-x_I|.
\end{equation}
Moreover, by Lemma \ref{koksi2} $|f''^{(n)}(\theta)|\geq q_s^{3-\eta}\geq \frac{n^{3-\eta}}{\log^4n}$ (the last inequality by diophantine assumptions on $\a$). Let $I_{bad}:=[-\frac{1}{n\log^{80\gamma^{-1}}n}+x_I,x_I+
\frac{1}{n\log^{80\gamma^{-1}}n}]
$.
Then by \eqref{spek}
$$
 (V^c_n\cap I)\subset I_{bad}.
$$

So
$$ 
V^c_n\subset \bigcup_{I\in \mathcal{I}}I_{bad}^f;
$$
and therefore  
$\mu(V^c_n)\leq \frac{C}{\log^{50\gamma^{-1}}n}$. This gives Lemma \ref{smd}.
\end{proof}

\begin{proof}[Proof of Lemma \ref{con}]
Since $|f^{(M)}(x)-M|\leq 2M^{1-\gamma}\log^{6}M$ it is enough to show that for 
$$R_M:=\{t\in [0,f^{(M)}(x)]\;:\;x\notin W_{N(x,t)}\}
$$ 
we have 
$$
\lambda(R_M)<\frac{M}{2\log^3M} 
$$
Divide the interval $[0,f^{(M)}(x)]$ into intervals of length $q_k+1$, where $q_k$ is the denominator closest to $\log^{4\gamma^{-1}}M$ (notice that the endpoints of the intervals are integers).  

Denote these intervals by $[N_j,N_{j+1}]$, $j=0,..., \left[\frac{f^{(M)}(x)}{q_k}\right]-1$. Notice that by \eqref{ase}, for  at least  $\left[\frac{f^{(M)}(x)}{q_k}\right]-\frac{M}{\log^{10\gamma^{-1}}M}$ of $j$'s we have
\begin{equation}\label{coh}
[N_j,N_{j+1}]\cap \{i\in[0,M]\;:\; x\notin W_i\}=\emptyset
\end{equation}
For any such $j$ let $T_j$ be such that $N(x,T_j)=N_j$. Then by \eqref{coh}, 
$$
[T_j,T_{j+1}]\cap R_M=\emptyset.
$$
By definition $N_{j+1}-N_j=q_k+1$. Therefore and by \eqref{koks3}, we have 
\begin{multline*}|T_{j+1}-T_j|\geq f^{(N(x,T_{j+1}))}(x)-f^{(N(x,T_j+1))}(x)=\\
f^{(N_{j+1}-N_j-1)}(x+(N_j+1)\alpha)=f^{(q_k)}(x+(N_j+1)\alpha)\geq  q_k-4q_k^{1-\gamma}.
\end{multline*}
So, finally
$$
\lambda(R_M)\leq f^{(M)}(x)- \left(\left[\frac{f^{(M)}(x)}{q_k}\right]-\frac{M}{\log^{10\gamma^{-1}}M}\right)
(q_k-4q_k^{1-\gamma})\leq
\frac{M}{2\log^3M},
$$
by the choice of $q_k$. This finishes the proof.
\end{proof}

\bibliographystyle{plain}
\bibliography{slow}

\end{document}